\RequirePackage{snapshot}
\documentclass[hidelinks,onefignum,onetabnum]{siamart220329}

\newtheorem{assumption}[theorem]{Assumption}
\usepackage{bookmark}
\usepackage{hyperref}

\newcommand{\N}{{\mathbb N}}

\newcommand{\mE}{\mathbb E}
\newcommand{\mR}{\mathbb R}
\newcommand{\mC}{\mathbb C}
\newcommand{\qall}{\quad\forall}
\newcommand{\cA}{{\mathcal A}}
\newcommand{\cM}{\mathcal M}
\newcommand{\mN}{\mathbb N}
\newcommand{\cG}{{\cal G}}
\newcommand{\cF}{\mathcal F}

\newcommand{\norm}[2]{\left\|{#1}\right\|_{#2}}
\newcommand{\brac}[1]{\left(#1\right)}
\newcommand{\sett}[1]{\left\{#1\right\}}
\newcommand{\gdnota}[1]{\left[\tfrac{1}{2} \right]_{#1} }
\newcommand{\abs}[1]{\left|#1\right|}
\newcommand{\bsb}[1]{\boldsymbol{#1}}
\newcommand{\floor}[1]{\lfloor #1 \rfloor}
\newcommand{\bga}{\bsb{\gamma}}

\newcommand{\eps}{\epsilon}
\newcommand{\vzero}{\boldsymbol{0}}
\newcommand{\vecxi}{\boldsymbol{\xi}}
\newcommand{\vupsilon}{{\boldsymbol{\upsilon}}}
\newcommand{\mfu}{\mathfrak{u}}
\newcommand{\mW}{\mathcal{W}}
\newcommand{\vx}{\boldsymbol{x}}
\newcommand{\vy}{\boldsymbol{y}}
\newcommand{\vz}{\boldsymbol{z}}
\newcommand{\vR}{\boldsymbol{R}}
\newcommand{\vell}{\boldsymbol{\ell}}
\newcommand{\vm}{\boldsymbol{m}}
\newcommand{\vzeta}{\boldsymbol{\zeta}}
\newcommand{\vnu}{\boldsymbol{\nu}}
\newcommand{\valpha}{\boldsymbol{\alpha}}
\newcommand{\vbeta}{\boldsymbol{\beta}}
\newcommand{\ve}{\boldsymbol{e}}
\newcommand{\spann}{{\rm{span}}}
\newcommand{\goto}{\rightarrow}
\newcommand{\wtd}{\widetilde}
\newcommand{\con}{K}

\usepackage{lipsum}
\usepackage{amsfonts}
\usepackage{graphicx}
\usepackage{epstopdf}
\usepackage{algorithmic}
\ifpdf
  \DeclareGraphicsExtensions{.eps,.pdf,.png,.jpg}
\else
  \DeclareGraphicsExtensions{.eps}
\fi

\newsiamremark{remark}{Remark}
\newsiamremark{hypothesis}{Hypothesis}
\crefname{hypothesis}{Hypothesis}{Hypotheses}
\newsiamthm{claim}{Claim}

\headers{ANALYTIC AND GEVREY REGULARITY FOR PARAMETRIC ELLIPTIC EVP}{Alexey Chernov AND T\`UNG L\^E}

\title{Analytic and Gevrey class regularity for parametric elliptic eigenvalue problems and applications\thanks{Submitted to the editors June 30, 2023.}}

\author{Alexey Chernov\thanks{Institut f\"ur Mathematik,	Carl von Ossietzky Universit\"at Oldenburg
		(\email{alexey.chernov@uni-oldenburg.de}, \email{tung.le@uni-oldenburg.de}).}
	\and
	T\`ung L\^e\footnotemark[2]}

\subtitle{\textit{Dedicated to Professor Christoph Schwab on the occasion of his 60th birthday}}

\usepackage{amsopn}

\ifpdf
\hypersetup{
  pdftitle={Analytic and Gevrey class regularity for parametric elliptic eigenvalue problems and applications},
  pdfauthor={Alexey Chernov and Tung Le}
}
\fi

\begin{document}

\maketitle

\begin{abstract}
We investigate a class of parametric elliptic eigenvalue problems with homogeneous essential boundary conditions, where the coefficients (and hence the solution) may depend on a parameter. For the efficient approximate evaluation of parameter sensitivities of the first eigenpairs on the entire parameter space we propose and analyse Gevrey class and analytic regularity of the solution with respect to the parameters. This is made possible by a novel proof technique, which we introduce and demonstrate in this paper. Our regularity result has immediate implications for convergence of various numerical schemes for parametric elliptic eigenvalue problems, in particular, for the Quasi-Monte Carlo integration.
\end{abstract}


\section{Introduction} \label{sec: intro}
Multiparametric eigenvalue problems arise in numerous applications in engineering and physics. For the perturbation theory we refer to the works by Rellich and Kato \cite{Kato:1966:PTL, rellich1969perturbation} and more recent works by Andreev and Schwab \cite{Andreev2012}, Gilbert et al. \cite{Gilbert2019, GilbertScheichl}, Grubi\v{s}i\'{c} et al. \cite{Hakula23}, and D\"olz and Ebert \cite{DoelzPrep22} where the particular case of stochastic parameter perturbation has been addressed. Irrespectively from the nature of the perturbation the study of the regularity of the solution is important since it allows to design appropriate approximation schemes in the parameter domain, tailored to a specific regularity class. For example, {the} articles \cite{Gilbert2019,GilbertScheichl} propose efficient Quasi-Monte Carlo approximations for elliptic eigenvalue problems with stochastic coefficients. The corresponding regularity analysis is a part of the convergence proof.

In this paper, we address the regularity of the solution to real-valued second-order elliptic eigenvalue problems (EVP) in the general form
\begin{equation}\label{gen equation}
	{\small
	\begin{split}
		-\nabla\cdot (\tilde a(\vx,\vy)\nabla u(\vx,\vy))
		+b(\vx,\vy)u(\vx,\vy)
		&=
		\lambda(\vy)
		c(\vx,\vy)u(\vx,\vy) \qquad
		(\vx,\vy) \in D \times U, 	\\
		u(\vx,\vy)&=0 \hspace*{31mm} 
		(\vx,\vy) \in \partial D \times U, \end{split}}
\end{equation}
where the derivative operator $\nabla$ acts in the physical variables $\vx \in D$, where $D \subset \mR^d$ is a bounded Lipschitz domain. The vector of parameters $\vy = (y_1,y_2,\dots) \in U$ has either finitely many or countably many components. For example, if $\vy$ is a random parameter, the model with $U:= [-\frac{1}{2},\frac{1}{2}]^{\mN}$ and $\vy \in U$ being a countably-dimensional vector of independently and identically distributed uniform random variables has been frequently used in the literature \cite{CohenDevoreSchwab2010,CohenDevoreSchwab2011,KuoSchwabSloan2013,KuoNuyens2016}. We mention that for our analysis boundedness of the parameter domain is necessary as it implies the positivity of the relative spectral gap introduced in Section \ref{sec:EVPtheory} below.

It turns out that it is more convenient to work with the rescaled version of the diffusion coefficient
\begin{equation}\label{def-a}
a(\vx, \vy) = \chi_1 \tilde a(\vx,\vy),
\end{equation}
where $\chi_1$ is the smallest eigenvalue of the deterministic Dirichlet-Laplace eigenvalue problem in the domain $D$
\begin{equation}\label{def-chi}
\chi_1 = \inf_{v \in H^1_0(D)} \dfrac{\int_D |\nabla u|^2 \, dx}{\int_D |u|^2 \, dx}.
\end{equation}
On the one hand, for a fixed domain $D$ the eigenvalue $\chi_1$ is a fixed positive number and therefore the rescaling \eqref{def-a} is no restriction of generality. On the other hand, this trivial rescaling helps to simplify the expressions in the forthcoming proofs. Moreover, we may argue that $\tilde{a}$ is the only parameter in \eqref{gen equation} which depends as $\mathcal L^2$ on the length units $\mathcal L$, whereas by \eqref{def-chi} $\chi_1$ is proportional to~$\mathcal L^{-2}$. In this sense, the rescaling \eqref{def-a} is natural, since it renders $a(\vx,\vy)$ invariant with respect to the change of length units. 

In the following we assume that the coefficients $a$, $b$, $c$ admit the uniform bounds
\begin{equation}\label{abc-bounds}
	\frac{\overline{a}}{2} \geq a(\vx,\vy)\geq  \underline{a}  >0, \qquad
	\frac{\overline{b}}{2}\geq b(\vx,\vy)\geq  0, \qquad 
	\frac{\overline{c}}{2}\geq c(\vx,\vy)\geq  \underline{c}   >0
\end{equation}
for all $\vy \in U$ and almost all $\vx \in D$. In this case for every fixed $\vy \in U$ the eigenvalues of~\eqref{gen equation} are real, positive, bounded from below, and the smallest eigenvalue $\lambda_1$ is simple, see \cite[Theorem 8.37]{GilbargTrudinger98} and \cite{Davies1995,Henrot2006}. Since the coefficients depend on the parameters~$\vy$, the eigenvalues $\lambda$ and corresponding eigenfunctions $u$ will depend on $\vy$ as well. Particularly, if $\vy$ is random, then $\lambda(\vy)$ and $u(\vx,\vy)$ will be {random, too.}

In this paper we present a rigorous regularity analysis for the smallest eigenvalue~$\lambda_1$ and the corresponding eigenfunction $u_1$ with respect to the parameter $\vy$ in the general case where the coefficients $a, b, c$ are infinitely differentiable functions of $\vy$ belonging to the Gevrey class $G^{\delta}$ for some fixed $\delta \geq 1$. The scale of Gevrey classes is a nested scale of the parameter $\delta$ that fills the gap between analytic and $C^\infty$  functions 
\[
{\mathcal A} \subseteq G^{\delta} \subset G^{\delta'} \subset C^\infty, \qquad 1 \leq \delta < \delta'. 
\]
The class of analytic functions $\mathcal A$, corresponding to the special case $G^\delta$ with $\delta = 1$, is the smallest and arguably most important of this scale (we refer to Section \ref{sec: Preliminary} for precise definitions) and has been addressed for parametric/stochastic EVP before. The work \cite{Andreev2012} proves the analyticity of the smallest eigenpair using elegant complex analysis arguments. More recently Gilbert et al. \cite{Gilbert2019,GilbertScheichl} have used inductive real-variable arguments to derive explicit upper bounds for all $\vy$-derivatives of the eigenpair that are close to bounds for real analytic functions, but are still include $\eps$-suboptimal terms in the factorial term. 
Moreover, all above mentioned works consider the special case of the affine parametrization of the coefficients of the type
\begin{equation}\label{a-affine}
	a(\vx,\vy) = a_0(\vx) + \sum_{m=1}^\infty y_m a_m(\vx).
\end{equation}
Let us also mention a specific periodic model 
\begin{equation}\label{a-periodic}
	a(\vx,\vy) = a_0(\vx) + \sum_{m=1}^\infty \sin(2 \pi y_m) a_m(\vx)
\end{equation}
introduced and analysed for parametric elliptic source problems in the work \cite{Kaarnioja2020}.

\newpage

In this paper we go beyond these settings in various ways. 
\begin{enumerate}
	\item We demonstrate that the reason for the suboptimality in the real-variable inductive proofs is indeed an artefact of the proof strategy. We introduce a modified argument and obtain optimal upper bounds for the $\vy$-derivatives of the smallest eigenvalue and the corresponding eigenfunction.
	\item We go beyond affine coefficient expansions as in \eqref{a-affine} or the periodic expansion \eqref{a-periodic} and allow for general analytic ($\delta = 1$) and Gevrey class ($\delta \geq 1$) dependency on $\vy$ in all coefficients $a, b, c$ in \eqref{gen equation}.
	\item The case of Gevrey class non-analytic coefficients ($\delta>1$) has not been considered in the existing literature (notice that complex analysis arguments cannot be applied in this case). {This case is important since} it greatly extends the palette of parametrizations that can be used to model various phenomena, covering e.g. partition-of-unity parametrizations {in $\vy$. In this case the coefficients $a,b,c$ in \eqref{gen equation} would be Gevrey-$\delta$ non-analytic functions in $\vy$.}	
\end{enumerate}
It is worth mentioning that the forthcoming analysis for the Gevrey non-analytic case $\delta>1$ is no more complicated than the proof for the analytic case $\delta = 1$. This follows from the rather elementary estimate $(n!\,m!)^{\delta-1}\leq ((n+m)!)^{\delta-1}$ for any natural numbers $n$ and $m$ and $\delta>1$. We refer to Section \ref{sec: main result} below for the details.

The proofs of new, more general regularity results become possible due to the novel proof technique that we call the \emph{alternative-to-factorial technique}. The aim of this paper is to introduce this approach and demonstrate it in the example of the parametric eigenvalue problems. However, our new proof technique is not limited to the class of EVPs considered here and can be applied to more general {nonlinear} parametric PDEs, and hence this is a subject of our current research.

The structure of the paper is as follows. In Section \ref{sec: Preliminary} we demonstrate the reasons for the deficiency of the real-value inductive arguments from the existing literature and explain how our novel \emph{alternative-to-factorial technique} helps to overcome the difficulties. In Section \ref{sec:AssumptionMainThm} we formulate the regularity assumptions on the coefficients of the eigenvalue problem \eqref{gen equation} and state our main theorem claiming the same type regularity for the smallest eigenvalue and the corresponding eigenfunction. In Section~\ref{sec:EVPtheory} we summarize the properties of elliptic eigenvalue problems needed for the forthcoming regularity analysis. In Section \ref{sec: main result} we present the proof of the main result. The meaning and the validity of the main regularity result is illustrated by numerical experiments in Section \ref{sec: num exp}. The appendix contains technical results required for the proofs, including results on the Quasi-Monte Carlo integration of high-dimensional Gevrey-$\delta$ functions being interesting on their own.

{In what follows} we denote by $\mN$ the set of positive and by $\mN_0$ of nonnegative integers. By $L^2(D)$ and $L^\infty(D)$ we denote the spaces of square integrable and bounded functions equipped with standard norms. The space $V:= H^1_0(D)$ is the subspace of~$L^2(D)$ consisting of functions with square integrable partial derivatives and having vanishing Dirichlet trace at the boundary $\partial D$. The space $V$ equipped with the norm
\begin{equation}\label{def-V-norm}
\|v\|_V := \chi_1^{-1/2} \|\nabla v\|_{L^2(D)},
\end{equation}
where $\chi_1$ is the smallest eigenvalue of the deterministic Dirichlet-Laplace problem \eqref{def-chi}. Similarly as in \eqref{def-a}, the multiplier $\chi_1$ is added here for scaling reasons. This will help us to simplify the expressions throughout the paper. Finally, $V^*:=H^{-1}(D)$ is the dual to $V$. We will also occasionally use the short notation $\norm{\cdot}{\infty}$ for the essential supremum norm $\norm{\cdot}{L^{\infty}(D)}$ in the physical domain $D$.

\newpage

\section{Preliminaries} \label{sec: Preliminary}
\subsection{Deficiency of the real-variable inductive argument} \label{sec: Deficiency}
The deficiency of the real-variable inductive argument for nonlinear problems can be seen already in the one-dimensional case. It is in fact a consequence of the Leibniz product rule combined with the triangle inequality. To illustrate this, we collect some elementary  results on real analytic functions on the real line, see e.g. \cite[Section 1]{Krantz2002}.

\begin{definition}
Let $U \subseteq \mR$ be an open domain. A function $f$ is called real analytic at $y_0 \in U$, if it can be represented by a convergent power series
	\begin{equation}\label{power-series-1d}
		f(y) = \sum_{j=0}^\infty a_j (y-y_0)^j
	\end{equation}
	in some neighbourhood of $y_0$. The function $f$ is called real analytic in an open subset $I \subseteq U$, if it is analytic in all $y_0 \in I$.
\end{definition}
The uniqueness of the power series implies the identity $a_n = f^{(n)}(y_0) / n!$ and (after some additional work) the following alternative characterization of real analytic functions.

\begin{theorem}\label{thm:ana-bound-1d}
	Let $f \in C^\infty(I)$ for some open interval $I$. The function $f$ is real analytic if and only if 
	for each $y_0 \in I$, {there is} an open interval $J$, with
	$y_0 \in J \subseteq I$, and constants $R > 0$ and $C > 0$ such that the derivatives of $f$ satisfy
	\begin{equation}\label{ana-bound-1d}
		\frac{|f^{(n)}(y)|}{n!} \leq \frac{C}{R^n}, \qquad \forall y \in J, \quad n \in \mN_0.
	\end{equation}
	The radius of convergence $\rho$ of the power series \eqref{power-series-1d} at some $y_0 \in I$ can be determined as the supremum of $R$ such that \eqref{ana-bound-1d} holds with some $C = C(R)$.
\end{theorem} 

Let now $f$ be some real analytic function normalized to $C = 1$, centred at $y$ with radius of convergence $\rho_f$ and consider $g = f^2$. Clearly, $g$ is analytic, but what is its radius of convergence $\rho_g$? If we use the Leibniz product rule, the triangle inequality and \eqref{ana-bound-1d}, we get
\begin{equation}\label{g2-argument}
	\frac{|g^{(n)}(y)|}{n!} \leq \frac{1}{n!}\sum_{m=0}^n \brac{n \atop m} |f^{(m)}(y)| \, |f^{(n-m)}(y)|
	\leq  \frac{n+1}{R^n}.
\end{equation}
since
\begin{equation}\label{g2-argument2}
	\sum_{m=0}^n \brac{n \atop m} m! (n-m)! = (n+1) n!
\end{equation}
Observe that the upper bound has increased by the factor $(n+1)$ and therefore does not fit the upper bound \eqref{ana-bound-1d} with the same $R$. One way to proceed is to use the crude estimate $n+1 \leq 2^n$ which holds for all $n \in \mN_0$, but according to the upper bound this implies that the radius of convergence~$\rho_g$ of $g$ degrades to $\rho_f/2$. This kind of argument clearly cannot be used in inductive proofs, since when $n \to \infty$ the radius of convergence will degenerate to zero. A more successful option is to replace the term $n!$ by $(n!)^{1+\eps}$ for $\eps > 0$, see e.g. the work by Gilbert et al. in \cite{Gilbert2019} for further details. In this case $\eps$-suboptimal terms pollute the upper bounds and analytic-type estimates as in \eqref{ana-bound-1d} cannot be obtained for $g$.

On the other hand, it is the classic result that the function $g=f^2$ is analytic with \emph{the same radius of convergence} as the function $f$ itself{, see e.g. \cite[Chapter 1, Proposition 1.1.7]{Krantz2002}}. This can either be seen from the root test applied to \eqref{g2-argument}, or by the following explicit bound for the remainder of the power series of $g$. 
{Recall the power series representation \eqref{power-series-1d} for $f$ and let $S_N(y) = \sum_{m=0}^N b_m (y-y_0)^m$ be the partial sum of $g=f^2$, where $b_m = \sum_{j+k = m} a_ja_k$. 
Then}{ 
\begin{equation*} 
\begin{split}
	|g(y) - S_{N}(y)| 
	&= \left| \sum_{m=N+1}^\infty\, \sum_{j+k = m} a_j a_k (y-y_0)^{j+k} \right| \\ 
	&\leq 2 \sum_{j=\floor{(N+1)/2}}^\infty |a_j| \, |y-y_0|^j \sum_{k=0}^\infty |a_k| \, |y-y_0|^k,
	\end{split}
\end{equation*}
where $\floor{.}$ is the floor function.} Since the power series of $f$ converges uniformly and absolutely for $|y-y_0|<\rho_f$, the $k$-sum is bounded and the $j$-sum tends to zero for $N \to \infty$ and all {$y$ satisfying} $|y-y_0| < \rho_f$. Hence $\rho_g = \rho_f$. This demonstrates that inductive proofs based on the Leibniz product rule and the triangle inequality in combination with \eqref{ana-bound-1d} can hardly lead to optimal regularity estimates for nonlinear parametric problems.  

\subsection{The falling factorial estimates as a remedy}\label{sec: Remedy}
Observe that the quadratic type nonlinearity $g = f^2$ discussed in the previous section is the simplest, yet very important type of nonlinearity appearing e.g. in Navier-Stokes equations and eigenvalue problems that serve as a model for this paper. The use of the Leibniz product rule together with the triangle inequality in the parametric regularity proofs is a desirable tool that splits the nonlinear part into basic contributions and therefore should be preserved. \emph{Instead we seek for an alternative equivalent for \eqref{ana-bound-1d} that can be used in the course of induction}. To get a better insight, let us understand whether there is a function $f$ defined in a neighbourhood $J$ of $y_0=0$, such that 
\begin{equation}\label{f2-cond}
	|g^{(n)}(y)| \leq |f^{(n)}(y)| \qquad \forall y \in J, \quad n \in \mN_0
\end{equation}
is valid for $g = f^2$? The answer is positive and an example of such a function $f$ can be given in closed form, that is
\begin{equation}\label{f2-func}
	f(y) = \tfrac{1}{2}(1-\sqrt{1-y}).
\end{equation}
Indeed, $g(y) = f(y) - \frac{y}{4}$ and hence \eqref{f2-cond} holds as equality for $n \geq 2$. Basic calculations also verify \eqref{f2-cond} for $n = 0$ and $1$ and any $J \subset [-3,1)$. The derivatives of $f$ are all positive for $y<1$ and satisfy
\begin{equation}\label{f2-fallfact}
	2f^{(n)}(y) = (-1)^{n-1}\left(\tfrac{1}{2}\right)\left(\tfrac{1}{2} -1 \right)\dots \left(\tfrac{1}{2}-n + 1\right) (1-y)^{\tfrac{1}{2}-n} = (-1)^{n-1} \left(\tfrac{1}{2}\right)_n (1-y)^{\tfrac{1}{2}-n}
\end{equation}
for $n \geq 1$ and the \emph{falling factorial} notation
\begin{equation}\label{def-fallfact}
	(q)_n := \left\{
	\begin{array}{cl}
		1, & n=0, \\
		q(q-1)\dots(q-n+1), & n \geq 1,
	\end{array}
	\right. \qquad q \in \mR, \quad n \in \mN_0.
\end{equation}
For $q<1$ the falling factorial $(q)_n$ is a sign-alternating sequence of $n$. As already indicated by \eqref{f2-func} and \eqref{f2-fallfact}, for our purposes the particular value $q=\frac{1}{2}$ will be important. To further simplify the notation and avoid keeping track of the sign alteration, we denote the \emph{absolute value of the falling factorial of $\frac{1}{2}$} by
\begin{equation*} 
	\gdnota{n}
	:=
	\left|\left(\tfrac{1}{2} \right)_{n}\right|.
\end{equation*}
This notation appears somewhat non-standard, but quite convenient, as we will see in the sequel. Notice that the defining property of the Gamma function $\Gamma(z+1) = z\Gamma(z)$ and $\Gamma(\tfrac{1}{2}) = \sqrt{\pi}$ imply the representation for all nonnegative integers $n \in \N_0$ 
\begin{equation}\label{fallfact-gamma}
	\left[ \tfrac{1}{2} \right]_{n} = \frac{1}{2\sqrt{\pi}}
	|\Gamma(n-\tfrac{1}{2})|
	\quad \text{or equivalently} \quad
	n! = \xi_n \gdnota{n} \quad \text{with} \quad \xi_n = 2\sqrt{\pi}\frac{\Gamma(n+1)}{|\Gamma(n - \tfrac{1}{2})|}
\end{equation}
and the trivial bound $1 \leq \xi_n \leq 2 \cdot 2^n$ implying
\begin{equation}\label{multiindex-est-2}
	[\tfrac{1}{2}]_{n} \leq n! \leq 2 \cdot 2^n [\tfrac{1}{2}]_{n}.
\end{equation}
Observe next that $g^{(n)}(0) = f^{(n)}(0)$ holds for $n \geq 2$ and therefore this value can be computed in two different ways: On the one hand, it satisfies \eqref{f2-fallfact}, i.e. $g^{(n)}(0) = \tfrac{1}{2}\gdnota{n}$. On the other hand, by the Leibniz product rule
\begin{equation}
	g^{(n)}(0) = \sum_{i=0}^n \brac{n \atop i} f^{(i)}(0)f^{(n-i)}(0) = \frac{1}{4}\sum_{i=1}^{n-1} \brac{n \atop i} \gdnota{i}\, 
	\gdnota{n-i}, \qquad n \geq 2
\end{equation}
where we used $f(0) = 0$ in the last step.
Since by definition $\gdnota{0} = 1$, this implies the following combinatorial identities.
\begin{lemma} \label{lem:ff-main-id}
	For all integers $n\geq 2$ the following identities hold
	\begin{equation*}\label{ff-main-id} 
		\footnotesize{
		\sum_{i=1}^{n-1} \binom{n}{i}  
		\gdnota{i}\, 
		\gdnota{n-i}
		=
		2
		\gdnota{n}, \quad
		\sum_{i=1}^{n} \binom{n}{i}  
		\gdnota{i}\, 
		\gdnota{n-i}
		=
		3
		\gdnota{n}, \quad
		\sum_{i=0}^{n} \binom{n}{i}  
		\gdnota{i}\, 
		\gdnota{n-i}
		=
		4
		\gdnota{n}.}
	\end{equation*}
	With the convention that the empty sum equals zero the statement extends to all integers $n \geq 0$ as 	
	\begin{equation*}
		\footnotesize{
		\sum_{i=1}^{n-1} \binom{n}{i}  
		\gdnota{i}\, 
		\gdnota{n-i}
		\leq
		2
		\gdnota{n}, \quad
		\sum_{i=1}^{n} \binom{n}{i}  
		\gdnota{i}\, 
		\gdnota{n-i}
		\leq
		3
		\gdnota{n}, \quad
		\sum_{i=0}^{n} \binom{n}{i}  
		\gdnota{i}\, 
		\gdnota{n-i}
		\leq
		4
		\gdnota{n}.}
	\end{equation*}
\end{lemma} 
Remarkably, the constants in the right hand sides of the above identities are uniformly bounded, whereas the constant in the right hand side of \eqref{g2-argument2} grows as $(n+1)$. This suggests to replace the term $n!$ by the falling factorial $[\frac{1}{2}]_n$ in the inductive proof. Moreover, there holds the following alternative characterization of analytic functions.

\begin{theorem}
	Let $f \in C^\infty(I)$ for some open interval $I$. The function $f$ is real analytic if and only if 
	for each $y_0 \in I$, there are an open interval $J$, with
	$y_0 \in J \subseteq I$, and constants $C > 0$ and $R > 0$ such that the derivatives of $f$ satisfy
	\begin{equation}\label{ff-bound-1d}
		|f^{(n)}(y)|\leq \frac{C}{R^n} \gdnota{n}, \qquad \forall y \in J, \quad n \in \mN_0.
	\end{equation}
	The radius of convergence $\rho$ of the power series \eqref{power-series-1d} at some $y_0 \in I$ can be determined as the supremum of $R$ satisfying \eqref{ff-bound-1d}.
\end{theorem} 

\begin{proof} 
	Let us fix $y_0 \in I$. By Theorem \ref{thm:ana-bound-1d} the radius of convergence of the power series \eqref{power-series-1d} at $y_0$ can be written as
	\begin{equation*}
		\rho = \sup\{R~:~ R \text{ satisfies \eqref{ana-bound-1d} with some constant } C = C(R) \}.
	\end{equation*}
	Thus, if we denote
	\begin{equation*}
		\tilde \rho = \sup\{R~:~ R \text{ satisfies \eqref{ff-bound-1d} with some constant } C = C(R) \}
	\end{equation*}
	our aim would be to show that $\rho=\tilde \rho$. On the one hand, there holds $\gdnota{n} \leq n!$. Together with \eqref{ff-bound-1d} this implies \eqref{ana-bound-1d} with the same $R$ and $C$ and thereby $\rho \leq \tilde \rho$. On the other hand, if \eqref{ana-bound-1d} holds, i.e. $f$ is analytic at $y_0 \in I$, the root test shows that the second derivative $f''$ in analytic at $y_0$ too and moreover has the same radius of convergence $\rho$ as the original function $f$, see e.g. \mbox{Proposition 1.1.13} in \cite{Krantz2002}. Now let $n \geq 2$.  By Theorem \ref{thm:ana-bound-1d} this means
	\begin{equation}\label{ddf-ana}
		|f^{(n)}(y)| = |(f''(y))^{(n-2)}(y)| \leq \frac{C}{R^{n-2}} (n-2)!
	\end{equation}
	From definition \eqref{fallfact-gamma} of the falling factorial we find $(n-2)! \leq 4 \gdnota{n}$ and hence \eqref{ddf-ana} implies \eqref{ff-bound-1d} with the same $R$ and $\wtd C := 4 C R^2$
	\begin{equation*}
		|f^{(n)}(y)|\leq \frac{\wtd C}{R^n} \gdnota{n}, \qquad \forall y \in J, \quad n \in \mN_0, \quad n \geq 2.
	\end{equation*} 
	By a suitable adjustment of $\wtd C$ the range of $n$ can be extended to cover the cases $n = 0$ and $n = 1$.
	In other words, \eqref{ana-bound-1d} implies \eqref{ff-bound-1d} with the same $R$ and possibly adjusted~$C$. This implies $\rho \geq \tilde \rho$ and thereby finishes the proof. 
\end{proof} 

\subsection{Gevrey-class and analytic functions} \label{sec: Gevrey}
The following standard multiindex notations, see e.g. \cite{BoutetKnee1967,CohenDevoreSchwab2010}, will be used in what follows.   
We denote the countable set of finitely supported sequences of nonnegative integers by
\begin{equation}\label{cF-def}
	\cF := 
	\left\{
	\vnu =(\nu_1,\nu_2,\dots)~:~ \nu_j\in \mN_0,
	\text{ and } \nu_j \neq 0
	\text{ for a finite number of } j
	\right\},  
	\end{equation}
where the summation $\valpha + \vbeta$, equality $\valpha = \vbeta$ and the partial order relation $\valpha \leq \vbeta$ of elements in $\valpha, \vbeta \in \cF$ are understood componentwise. Moreover, $\valpha < \vbeta$ means $\valpha \leq \vbeta$ and $\valpha \not = \vbeta$. We write
\begin{align*}
	\abs{\vnu}
	:=
	\sum_{j\geq 1} \nu_j, 
	\qquad \qquad
	\vnu!
	:=
	\prod_{j\geq 1} v_j!, 
	\qquad \qquad
	\vR^{\vnu}= \prod_{j\geq 1} R_j^{v_j}
\end{align*}
for the absolute value, the multifactorial and the power with the multi-index $\vnu$ and a sequence $\vR=\{R_j\}_{j\geq 1}$ of positive real numbers. For two multiindices $\vm, \vnu \in \cF$ with $\vm \leq \vnu$ the binomial coefficient is defined by 
\begin{equation*}
\brac{\vnu \atop \vm} = \frac{\vnu !}{(\vnu - \vm)! \vm!}.
\end{equation*}
Notice that $\abs{\vnu}$ is finite if and only if $\vnu\in \cF$. For $\vnu\in \cF$ supported in $\sett{1,2,\dots, n}$, we define the partial derivative with respect to the variables $\vy$
\begin{align*}
	\partial^{\vnu} u
	=
	\frac{\partial^{\abs{\vnu}}u}
	{\partial y_1^{\nu_1} \partial y_2^{\nu_2} \dots \partial y_n^{\nu_n}}.
\end{align*}

The above multiindex notations are helpful for treatment of multiparametric objects. In particular, our regularity considerations are based on the following definition of Gevrey-$\delta$ functions with countably many parameters.

\begin{definition} \label{def:G-delta-def}
	Let $\delta \geq 1$, $B$ be a Banach space, $I \subset \mR^\mN$ be an open domain and a function $f:I \to B$ be such that its $\vy$-derivatives $\partial^{\vnu} f:I \to B$ are continuous for all $\vnu \in \cF$. We call the function $f$ Gevrey-$\delta$ if for each $y_0 \in I$ there exist an open set $J\subseteq I$ and $y_0 \in J$, and strictly positive constants $\vR = (R_1,R_2,\dots) \subset \mR^{\mN}$ and $C \in \mR$ that the derivatives of $f$ satisfy the bounds
	\begin{equation}\label{G-delta-def}
		\|\partial^{\vnu} f(\vy)\|_B \leq \frac{C}{\vR^{\vnu}} (|\vnu|!)^\delta, \qquad \forall \vy \in J, \quad \forall \vnu \in \cF.
	\end{equation}
	In this case we write $f \in G^\delta(U,B)$.
\end{definition}

\begin{remark}\label{rem:def-G-delta-is-new}
	For the case of finitely many parameters $\vy = (y_1,\dots,y_M)$ the definition of Gevrey-$\delta$ functions with values in $B = \mR$ or $\mC$ can e.g. be found in \cite[Section~1.4]{Rodino1993} as well as some further characterizations of this class in the case $M < \infty$. For the case $\delta=1$ and $M=\infty$, we refer to the rescaling argument from \cite[Section~2.2]{Gilbert2019}. However, to the best of authors' knowledge, characterizations of the type \eqref{G-delta-def} for $M = \infty$, likewise the related characterization  \eqref{abc tilder assumption} below (neither for $M<\infty$ nor $M = \infty$, or in the special case $\delta = 1$) have not been discussed  in the literature before.
\end{remark}

\begin{remark}
	Definition \ref{def:G-delta-def} is also suitable for the case of finitely many parameters. In particular, when $\vy = (y_1,\dots,y_M)$, $B = \mR$ or $\mC$ and $\delta=1$, the bound \eqref{G-delta-def} guarantees convergence of the power series of $f$ and therefore characterizes the class of analytic functions of $M$ variables, see e.g. in \cite[Section 2.2]{Krantz2002}. This requires the bound $|\vnu|! \leq M^{|\vnu|} \vnu!$ that is valid for a multiindex $\vnu$ with $M$ nonzero components.
\end{remark}

\begin{remark}
	In general, the estimate \eqref{G-delta-def} does not guarantee convergence of the power series. Indeed, in the case $\delta=1$ and countably many parameters \linebreak $\vy = (y_1,y_2,\dots)$, we need an additional assumption, for example, that the sequence $\vR = (R_1,R_2,\dots)$ grows such that $\sum_{i=1}^{\infty} R_i^{-1}<\infty$, cf. \cite{Gilbert2019}. To see this, observe that for all $\vy\in U$ from the neighbourhood \mbox{$\|\vy - \vy_0\|_\infty < r$} \eqref{G-delta-def} with $\delta = 1$ yields
\begin{align*}
	|f(\vy)|
	=
	\left|\sum_{\vnu\in \cF} \frac{\partial^{\vnu}f(\vy_0)}{\vnu !}(\vy-\vy_0)^{\vnu} \right|
	\leq
	C \sum_{\vnu\in \cF} \frac{\abs{\vnu}!}{\vnu !} \brac{\frac{r}{\vR}}^{\vnu}
	=
   \frac{C}{1-r\sum_{i=1}^{\infty} R_i^{-1}}, 
\end{align*}
where we have used the identity $\brac{1-\sum_{i=1}^\infty z_i}^{-1}=\sum_{\vnu\in \cF} \frac{\abs{\vnu}!}{\vnu !} \vz^{\vnu} $ in the last step.

\end{remark}

\section{Regularity assumptions and the formulation of the main result}\label{sec:AssumptionMainThm}

We are now ready to formulate the Gevrey-class regularity assumption on the coefficients $a, b, c$ of the eigenvalue problem \eqref{gen equation} which (as we prove in the following) will imply the Gevrey-class regularity of the eigenpairs with the same $\delta \geq 1$.

	\begin{assumption} \label{Assumption}
		For all values $\vy\in U$ and $\vnu \in \cF$, the coefficients $a(\vx,\vy), b(\vx,\vy)$ and $c(\vx,\vy)$ are of Gevrey class $G^\delta(U,L^{\infty}(D))$, that is there exist a sequence $\vR = (R_1, R_2, \dots)$ with positive components $R_i > 0$ and constants $\overline{a}, \overline{b}, \overline{c}$ independent of $\vnu$ such that
		\begin{equation}\label{abc assumption}
		\begin{split}
			\norm{\partial^{\vnu} a(\cdot,\vy)}{\infty}
			\leq
			\frac{\overline{a} }{2}
			\frac{(\abs{\vnu}!)^\delta  }{(2\vR)^{\vnu}}, 
\quad
			\norm{\partial^{\vnu} b(\cdot,\vy)}{\infty}
			\leq
			\frac{\overline{b} }{2}
			\frac{(\abs{\vnu}!)^\delta  }{(2\vR)^{\vnu}}, 
\quad
			\norm{\partial^{\vnu} c(\cdot,\vy)}{\infty}
			\leq
			\frac{\overline{c} }{2}
			\frac{(\abs{\vnu}!)^\delta  }{(2\vR)^{\vnu}}.
	    \end{split}
		\end{equation} 
	\end{assumption}
Notice that for $\vnu = 0$ assumption \eqref{abc assumption} agrees with the upper bounds in \eqref{abc-bounds}. Notice also that the components of $\vR$ are readily scaled by the factor of $2$. This leads to no loss of generality, but allows to shorten the forthcoming expressions. 
For example, in view of \eqref{multiindex-est-2} assumption \eqref{abc assumption} immediately implies
\begin{equation}\label{abc tilder assumption}
\max\bigg(
\frac{\norm{\partial^{\vnu} a(\cdot,\vy)}{\infty}}{\overline{a}},
\frac{\norm{\partial^{\vnu} b(\cdot,\vy)}{\infty}}{\overline{b}},
\frac{\norm{\partial^{\vnu} c(\cdot,\vy)}{\infty}}{\overline{c}}
\bigg) \leq \frac{\gdnota{\abs{\vnu}} }{\vR^{\vnu}(\abs{\vnu}!)^{1-\delta}}
\end{equation} for all $\vy \in U$ and $\vnu \in \cF$. 
{We remark that for better readability and where it is unambiguous we will write the term $(\abs{\vnu}!)^{1-\delta}$ in the denominator, as in \eqref{abc tilder assumption}. This will be particularly useful in Section \ref{sec: main result} below.}

The main result of this paper is given by the following theorem.

\begin{theorem}\label{gevrey regularity for eigenpairs for general equ}
	Let $\delta\geq 1$, $\vnu\in \cF$ and Assumption \ref{Assumption} be valid. There exist constants $C_\lambda, C_u>0$ and a sequence $\wtd \vR= (\wtd R_1, \wtd R_2, \dots)$ with positive components $\wtd R_i > 0$ such that for all $\vy\in U$ the derivative of the smallest eigenvalue of \eqref{variational transform} is bounded by
	\begin{align}\label{lambda bound}
		\abs{\partial^{\vnu}\lambda_1(\vy)}
		\leq
		C_{\lambda}
		\frac{(\abs{\vnu}!)^\delta}{ {\wtd\vR}^{\vnu}}		
	\end{align}
	and the norm of the derivative of the corresponding eigenfunction is bounded by
	\begin{align}\label{u V bound}
		\norm{\partial^{\vnu} \nabla u_1(\vy)}{L^2(D)}
		\leq
		C_u
		\frac{(\abs{\vnu}!)^\delta}{  {\wtd\vR}^{\vnu}}.
	\end{align}
\end{theorem}
\begin{proof}
Recall that $\norm{v}{V}= \chi_1^{-1/2}\norm{\nabla v}{L^2(D)}$. Then the statement follows directly from Theorem~\ref{gevrey regularity for eigenpairs for ref domain} from Section \ref{sec: main result} below by using $\gdnota{\abs{\vnu}} \leq |\vnu|!$ and the constants  {$C_\lambda=\overline{ \lambda_1}\sigma / \rho$, $C_u=\chi_1^{\frac{1}{2}} \overline{ u_1}\sigma / \rho$ and $\wtd R_i = R_i/ \rho$} with the scaling factors $\sigma,\rho \geq 1$ from Theorem~\ref{gevrey regularity for eigenpairs for ref domain}.
\end{proof}
\begin{remark}
{Theorem \ref{gevrey regularity for eigenpairs for general equ} implies that the eigenpairs share the same Gevrey-$\delta$ or analytic regularity with the general (non-affine) coefficient parametrizations. In particular, this regularity result determines the convergence rates of the Gauss-Legendre quadrature. In the forthcoming Section \ref{sec: Gauss-Legendre Quadrature} we demonstrate this argument for a particular EVP and provide numerical examples to support the theory. In the same spirit, Theorem \ref{gevrey regularity for eigenpairs for general equ} together with the forthcoming result in Lemma \ref{QMC theorem} also  guarantees  the convergence rate of the QMC method as in \eqref{errorQMC} and \eqref{QMC convergence rate}. We refer to Section \ref{sec: QMC} for the detailed theoretical argument and the numerical experiments.}
\end{remark}

\section{Elliptic eigenvalue problems with countably many parameters}\label{sec:EVPtheory}
In this section we collect the required notations and facts from the theory of variational eigenvalue problems. 
Throughout the paper, when it is unambiguous we will drop the $\vx$-dependence when referring to a function defined on $D$ at a parameter value $\vy$. For example, we will write $a(\vy) := a(\cdot, \vy)$ and analogously for the other coefficients and eigenfunctions.

For a fixed $\vy \in U$ the variational formulation of \eqref{gen equation} reads

\begin{equation}\label{variational transform}
	\chi_1^{-1}\int_D a(\vy)\nabla u(\vy)\, \nabla v
	+
	\int_D
	b(\vy)u(\vy)\, v
	=
	\lambda(\vy)
	\int_D
	c(\vy)\,u(\vy)\, v {,\qall v\in V},
\end{equation}
where $\chi_1$ is the smallest Dirichlet-Laplace eigenvalue defined in \eqref{def-chi} being related to the Poincar\'e constant. For the rescaled norm \eqref{def-V-norm} the Poincar\'e inequality reads
\begin{equation}\label{simpler Poincare iequ}
	\norm{v}{L^2(D)}\leq \chi_1^{-1/2} \norm{\nabla v}{L^2(D)}=\norm{v}{V}.
\end{equation}

For a fixed parameter value $\vy \in U$ we define the symmetric bilinear forms $\cA_{\vy}:V\times V\goto \mR$ and $\cM_{\vy}:L^2(D)\times L^2(D)\goto \mR$ as
\begin{equation}\label{AM-def}
	\cA_{\vy}(w,v)
	:=
	\chi_1^{-1}\int_{D} a(\vy)\nabla w \, \nabla v
	+\int_{D}
	b(\vy)w\, v, \qquad 
	\cM_{\vy}(w,v):=
	\int_{D}
	c(\vy)w\, v.
\end{equation}
In view of the bounds \eqref{abc-bounds}, $\cA_{\vy}$ and $\cM_{\vy}$ are inner products on $V$ and $L^2(D)$, satisfying
\begin{equation}\label{AM-bounds}
\begin{split}
\underline{a}\|v\|_{V}^2 \leq \cA_{\vy} (v,v)
\leq 
\frac{\overline{a}+\overline{b}}{2}\|v\|_{V}^2,& \qquad v \in V, \\
\underline{c}\| v\|_{L^2(D)}^2 \leq \cM_{ \vy} (  v, v)
\leq 
\frac{\overline{c}}{2} \| v\|_{L^2(D)}^2, &\qquad v \in L^2(D).
\end{split}
\end{equation} 
Thus, for a fixed $\vy\in U$, the variational {counterpart} of \eqref{gen equation} is the problem of finding a nontrivial $u(\vy)\in V$ and $\lambda(\vy)\in \mR$ such that
\begin{equation}\label{eigen prob}
\begin{split}
\cA_{\vy}(u(\vy),v)
&=
\lambda(\vy)
\cM_{\vy}(u(\vy),v),
\quad \forall v\in V,\\
\norm{u(\vy)}{\cM_{\vy}}&=1,
\end{split}
\end{equation}
where the eigenvectors are normalized with respect to the norm $\norm{ v}{\cM_{\vy}}=\sqrt{\cM_{\vy}(v,v)}$. 
Similarly as the Dirichlet-Laplacian, for a fixed $\vy\in U$ the formulation \eqref{eigen prob} has countably-many eigenvalues $\{\lambda_k(\vy)\}_{k\in \mN}$, which are positive, have finite multiplicity and accumulate only at infinity, the smallest eigenvalue is simple, cf. \cite{GilbargTrudinger98,Davies1995,Henrot2006}. Counting multiplicities, we enumerate them in the non-decreasing order as
\begin{align*}
	0< \lambda_1(\vy)
	<
	\lambda_2(\vy)
	\leq
	\lambda_3(\vy)
	\leq \dots
\end{align*}
with $\lambda_k(\vy)\goto \infty$ as $k\goto \infty$. For an eigenvalue $\lambda_k(\vy)$, we define its eigenspace to be
\begin{equation*}
	E(\lambda_k(\vy))
	:=
	\spann \sett{u_k\, : \, u_k \text{ is an eigenfunction corresponding to } \lambda_k(\vy)}.
\end{equation*} 
Since $\lambda_1$ is simple, $E(\lambda_1(\vy))$ is one-dimensional and $u_k(\vy)$ belong to the orthogonal complement
\begin{equation*}
E(\lambda_1(\vy))^{\bot}:=\big\{v~:~\cM_{ \vy}(v,u_1(\vy))=0\big\}.
\end{equation*} 
The min-max principle and \eqref{AM-bounds} imply uniform upper bounds for the smallest eigenvalue
\begin{equation}
	\lambda_1(\vy)
	=
	\inf_{v \in V} 
	\frac{\cA_{\vy}(v,v)}{\cM_{\vy}(v,v)}
	\leq
	\frac{\overline{a}+\overline{b}}{2\,\underline{c}}
	=:
	\overline{\lambda_1} \label{lambda_k bound}
\end{equation}
and the corresponding eigenfunction
\begin{equation}
	\norm{u_1(\vy)}{V}
	\leq
	\sqrt{\frac{\lambda_1(\vy)}{\underline{a}}}
	\leq
	\sqrt{\frac{\overline{a}+\overline{b}}{2\,\underline{c}\, \underline{a}}}
	=:
	\overline{u_1}.\label{u_k bound}
\end{equation}

\begin{remark}

The assumption that $b(\vy) \geq 0$ in \eqref{abc-bounds} is made without loss of generality since any EVP with $b < 0$, but satisfying the rest of Assumption \ref{Assumption}, can be shifted to an equivalent problem with a non-negative coefficient ``new $b$'' by adding {the term} $\sigma \cdot c(\vy) \cdot u(\vy)$ to the both sides of \eqref{gen equation}, where $\sigma$ is chosen such that $-b(\vy) \leq \sigma \cdot c(\vy)$ for all $\vx, \vy$. Such $\sigma$ exists due to Assumption \ref{Assumption}. The eigenvalues of the original EVP will become the eigenvalues of the shifted problem shifted by $-\sigma$, and the corresponding eigenspaces are unchanged. This has been used previously e.g. in \cite{Giani2011, Gilbert2019}.

\end{remark}

We denote by $\mu$ the uniform lower bound of the \emph{relative spectral gap} $1-\frac{\lambda_1(\vy)}{\lambda_2(\vy)}$ over all $\vy \in U$ 
 \begin{equation}\label{def-mu}
 \mu:=\min_{\vy\in U} \brac{1-\frac{\lambda_1(\vy)}{\lambda_2(\vy)}}.
 \end{equation}
{The Krein-Rutman theorem guarantees that the smallest eigenvalue is simple, i.e. \mbox{$0 < \lambda_1(\vy)<\lambda_2(\vy)$} for all $\vy\in U$. Moreover, Lemma \ref{lem:EV-continuous} below guarantees that the eigenvalues $\lambda_k(\vy)$ are continuous in $\vy$, and therefore the \emph{relative spectral gap} $1-\frac{\lambda_1(\vy)}{\lambda_2(\vy)}$ is also continuous.} Since $U$ is a bounded domain, the minimum is attained, cf. \cite[Proposition~1]{CuiDeSterckGilbertPolishchukScheichl2023}, and therefore $\mu$ is strictly positive, see also \cite{Andreev2012} and \cite{Gilbert2019}. Since the eigenvalues are positive, this implies $0 < \mu < 1$. {
\begin{lemma} \label{lem:EV-continuous}
Suppose that the coefficients $a(\vy)$, $b(\vy)$, $c(\vy)$ are continuous at some parameter value $\vz$, i.e.
$\forall \varepsilon>0$ $\exists d >0$ such that $\|\vy - \vz\| < d$ 
implies 
\begin{equation}\label{a b c continuous}
\norm{a(\vy)-a(\vz)}{\infty}+\norm{b(\vy)-b(\vz)}{\infty}+\norm{c(\vy)-c(\vz)}{\infty}
<
\varepsilon.
\end{equation}
Then the $k$-th eigenvalue $\lambda_k(\vy)$ is continuous at $\vy = \vz$, i.e.
\begin{equation}\label{lambda continuous}
\forall  \varepsilon_\lambda>0 \quad \exists d_{\lambda} >0: \quad \|\vy - \vz\| < d_{\lambda}
\quad \Rightarrow \quad |\lambda(\vy) - \lambda(\vz)| < \varepsilon_\lambda.
\end{equation}
\end{lemma}
\begin{proof} Let $\varepsilon_\lambda>0$ be arbitrary. We will select $\varepsilon$ and thereby $d_{\lambda}:=d$ as a function of $\varepsilon_\lambda$ so that \eqref{lambda continuous} holds.
Observe that estimates \eqref{simpler Poincare iequ}, \eqref{AM-bounds}, and \eqref{a b c continuous} imply for all $v\in V$ 
\begin{align}\label{continuity AM}
\abs{\cA_{\vy}(v,v)-\cA_{\vz}(v,v)}< \varepsilon\norm{v}{V}^2,
\quad
\abs{\cM_{\vy}(v,v)-\cM_{\vz}(v,v)}<\varepsilon\norm{v}{L^2(D)}^2,
\end{align}
and
\begin{align}\label{quotients bound}
 \norm{v}{V}^2  \leq\underline{a}^{-1}\cA_{\vy}(v,v), \qquad
 \norm{v}{L^2(D)}^2 \leq\underline{c}^{-1}\cM_{\vy}(v,v).
\end{align}
Without loss of generality, we assume $\lambda_k(\vy)\geq\lambda_k(\vz)$ and that $\varepsilon$ is sufficiently small such that $\varepsilon\leq \frac{\underline{c}}{2}$. Recalling the min-max principle (e.g. \cite[eq (2.8)]{Babuska1989}), we get
\begin{align*} 
\lambda_k(\vy)
&=
\min_{S_k\subset V\atop \dim S_k = k}
\max_{v\in S_k \atop v\neq 0 } 
\frac{\cA_{\vy}(v,v)}{\cM_{\vy}(v,v)}
\leq
\min_{S_k\subset V\atop \dim S_k = k}
\max_{v\in S_k \atop v\neq 0 } 
\frac{\cA_{\vz}(v,v)+\abs{\cA_{\vy}(v,v)-\cA_{\vz}(v,v)}}{\cM_{\vz}(v,v)-\abs{(\cM_{\vy}(v,v)-\cM_{\vz}(v,v)}}\\
&<
\min_{S_k\subset V\atop \dim S_k = k}
\max_{v\in S_k \atop v\neq 0 } 
\frac{\cA_{\vz}(v,v)+\varepsilon\norm{v}{V}^2}{\cM_{\vz}(v,v)-\varepsilon\norm{v}{L^2(D)}^2}
\leq
\min_{S_k\subset V\atop \dim S_k = k}
\max_{v\in S_k \atop v\neq 0 } 
\frac{\cA_{\vz}(v,v)}{\cM_{\vz}(v,v)}
\brac{
\frac{1+\underline{a}^{-1}\varepsilon}{1-\underline{c}^{-1}\varepsilon}}\\
&=
\brac{
\frac{1+\underline{a}^{-1}\varepsilon}{1-\underline{c}^{-1}\varepsilon}}
\lambda_k(\vz)
\leq
\lambda_k(\vz)+2\varepsilon\brac{\underline{a}^{-1}+\underline{c}^{-1}}\lambda_k(\vz).
\end{align*}
Observe that since $\varepsilon\leq \frac{\underline{c}}{2}$, the denominator $\cM_{\vzeta}(v,v)-\varepsilon\norm{v}{L^2(D)}^2$ is strictly positive, and $1-\underline{c}^{-1}\varepsilon\geq \frac{1}{2}$. 
Moreover, $\lambda_k(\vz)$ is uniformly bounded in $U$. Indeed, the min-max principle and \eqref{AM-bounds} imply for all 
$\vz\in U$ that
\begin{align*}
\lambda_k(\vz)
=
\min_{S_k\subset V\atop \dim S_k = k}
\max_{v\in S_k \atop v\neq 0 } 
\frac{\cA_{\vz}(v,v)}{\cM_{\vz}(v,v)}
\leq
\brac{
\frac{\overline{a}+\overline{b}}{2\underline{c}}}
\min_{S_k\subset V\atop \dim S_k = k}
\max_{v\in S_k \atop v\neq 0 } 
\frac{\norm{v}{V}^2}{\norm{v}{L^2(D)}^2}
=:\overline{ \lambda_k}.
\end{align*}
Notice that for all $k\in \mN$ the constants $\overline{\lambda_k}$ are independent of $\vz\in U$. We select $\varepsilon$ in \eqref{a b c continuous} as $\varepsilon=\min\sett{(2\brac{\underline{a}^{-1}+\underline{c}^{-1}}\overline{ \lambda_k})^{-1}\varepsilon_\lambda,\frac{\underline{c}}{2}}$. Thus, there exists $d_{\lambda}>0$ satisfying \eqref{lambda continuous}. This completes the proof.
\end{proof}
}

The following coercive-type estimate is required in order to control the norm of the $\vy$-derivatives of the eigenfunction $u_1(\vy)$.  The proof essentially follows the lines of \cite[Lemma 3.1]{Gilbert2019}, but covers the case of a parameter-dependent coefficient~$c=c(\vy)$.
\begin{lemma}\label{lem:coercive-type}
	Let $\mu$ be the relative spectral gap defined in \eqref{def-mu} and the coefficients $a$, $b$, $c$ satisfy the bounds \eqref{abc-bounds}. The shifted bilinear form $\wtd \cA_{\vy} := \cA_{\vy} - \lambda_1(\vy)\cM_{\vy}$	
is uniformly coercive on the orthogonal complement to $u_1(\vy)$, that is	
	\begin{equation*}
\cA_{\vy}(v,v)
		-
		\lambda_1(\vy)\cM_{\vy}(v,v)\geq
		C_{\mu} \norm{v}{V}^2,
		\qquad \forall v\in E(\lambda_1(\vy))^{\bot}
	\end{equation*}
with the coercivity constant 	$C_{\mu}= \underline{a}\,\mu$ being independent of $\vy$.
\end{lemma}
\begin{proof}
The eigenfunctions $\{u_k(\vy)\}_{k\in \mN}$ form a basis in $V$, that is orthonormal with respect to the inner product $\cM_{\vy}$. Therefore $v\in E(\lambda_1(\vy))^\bot$ admits the representation
\begin{align*}
v(\vy)=\sum_{k=2}^\infty {v_k(\vy)},
\end{align*}
where $v_k(\vy):=\cM_{\vy}(v,u_k(\vy)) u_k(\vy)$.
Notice that $v_1(\vy)=0$ since $v\in E(\lambda_1(\vy))^\bot$. We will suppress the dependence of the $v$ and $v_k$ on $\vy$ until the end of this proof for brevity. For $v\in E(\lambda_1(\vy))^\bot$, we have
\begin{align*}
\wtd \cA_{\vy}(v,v)
=
\sum_{k=2}^\infty
\sum_{j=2}^\infty
\big(
\cA_{\vy}\brac{v_k, v_j}-\lambda_1(\vy) \cM_{\vy}(v_k,v_j)
\big)
\end{align*}
Since all $v_k$ are just scaled versions of $u_k$, they also satisfy the variational equation \eqref{variational transform}, so that
$\cA_{\vy}(v_k,v_j)=\lambda_k\cM_{\vy}(v_k,v_j)$ 
and they are orthogonal with respect to $\cM_{\vy}(\cdot,\cdot)$. This implies that $\cA_{\vy}(v_k,v_j)=0$ for $k\neq j$. Thus, the above double sum is reduced to
\begin{align*}
\wtd \cA_{\vy}(v,v)
&=
\sum_{k=2}^\infty
\brac{
\cA_{\vy}\brac{v_k, v_k}-\frac{\lambda_1(\vy)}{\lambda_k(\vy)}\cA_{\vy}\brac{v_k, v_k}
}\\
&\geq
\brac{1-\frac{\lambda_1(\vy)}{\lambda_2(\vy)}}
\sum_{k=2}^\infty
\cA_{\vy}\brac{v_k, v_k}\\
&=
\brac{1-\frac{\lambda_1(\vy)}{\lambda_2(\vy)}}
\cA_{\vy}\brac{v, v}
\geq
\underline{a}\, \mu
\norm{v}{V}^2
=
C_{\mu}
\norm{v}{V}^2.
\end{align*}
This finishes the proof.
\end{proof}
\section{Parametric regularity: Proof of the main result}\label{sec: main result}
The aim of this section is the proof of the main  regularity result. As we will see, the scaling of the constants in the estimates for the $\vnu$-th derivative of the eigenpairs will \emph{solely depend on the contrast} of the coefficients of the eigenvalue problem \eqref{gen equation} defined as follows, cf. \eqref{AM-bounds}
\begin{equation}\label{def-contrast}
\con_a := 
 \frac{\overline{a} + \overline{b}}{2 \underline{a}},\qquad
\con_c := \frac{\overline{c}}{2 \underline{c}}.
\end{equation}
Notice that $\con_a, \con_c \geq 1$ and the following particular relations of dimensionless quantities
\begin{equation}\label{rel-contrast}
\frac{\overline{u_1}^2\overline{c} }{2} = \con_a \con_c = \frac{\overline{\lambda_1}\overline{c}}{2 \underline{a}},
\qquad
\frac{\overline{a} + \overline{b}  + \overline{\lambda_1}\, \overline{c}}{2\underline{a}} = \con_a (1 + \con_c).
\end{equation}

\begin{theorem}\label{gevrey regularity for eigenpairs for ref domain}
	Let $\delta\geq 1$, $\vnu\in \cF$, $|\vnu|\geq 1$ and $\vR = (R_1,R_2,\dots)$ be defined as in Assumption \ref{Assumption}. Then there exist $\sigma, \rho\geq 1$ independent of $\vnu$ such that the smallest eigenvalue $\lambda_1$ and the corresponding eigenfunction $u_1$ of \eqref{eigen prob} admit the bounds
{
	\begin{equation}\label{lambda bound ref}
		\abs{ \partial^{\vnu} \lambda_1( \vy)}
		\leq
\frac{\overline{\lambda_1}{\sigma}}{\rho} 
\left(\frac{ \rho}{\vR} \right)^{\vnu}  \frac{\gdnota{\abs{\vnu}}}{(\abs{\vnu}!)^{1-\delta}} 
	\end{equation}
	and
	\begin{equation}\label{u V bound ref}
		\norm{ \partial^{\vnu}  u_1(\vy)}{ V}
		\leq
\frac{\overline{u_1}{\sigma}}{\rho} 
\left(\frac{ \rho}{\vR} \right)^{\vnu}  \frac{\gdnota{\abs{\vnu}}}{(\abs{\vnu}!)^{1-\delta}}.
	\end{equation}
The scaling factors $\sigma$, $\rho$ depend solely on the contrast of the coefficients \eqref{def-contrast} and are defined as 
\begin{equation}\label{sigma rho def}
\begin{array}{ll}
	\sigma:=
		\mu^{-1}\sigma_1 + \con_a \con_c,&\quad 
\sigma_1
	:=
	2 \con_a (1+\con_c),
	\\[2ex]
	\rho:=
	\mu^{-1}\rho_1 + {(3+8\sigma)}\con_a \con_c,&\quad
\rho_1
	:=
	3 \sigma_1 + 16{\sigma} \con_a \con_c.
\end{array}
\end{equation}}
\end{theorem}
The proof of Theorem \ref{gevrey regularity for eigenpairs for ref domain} requires upper bounds for the derivatives of the eigenvalue and the eigenfunction collected in the following Lemma \ref{lem:Eval-bnd}, Lemma \ref{upper bound for eigenfunction 1st} and Lemma~\ref{upper bound for eigenfunction 2nd}.
\begin{lemma}\label{lem:Eval-bnd}
	For sufficiently regular solutions of \eqref{variational transform} and $\vnu \in \cF$ with $|\vnu| \geq 1$ there holds
	\begin{equation}\label{eigenvalue upper}
		{\normalsize
		\begin{split}
			\abs{ \partial^{\vnu}  \lambda_1}
			&\leq
			\bigg(
			\sum_{\vzero< \vm \leq \vnu}\brac{\vnu \atop \vm} 
			\norm{   \partial^{\vnu-\vm}  u_1 }{ V}  
			\brac{
			\norm{   \partial^{\vm}  a }{\infty} 
			+
			\norm{   \partial^{\vm}  b }{\infty}
			+ 
			\lambda_1
			\norm{   \partial^{\vm}  c }{\infty} }\\
			&\quad 
			+  
			\sum_{\vzero<\vm < \vnu}
			\brac{\vnu \atop \vm}
			\abs{ \partial^{\vnu-\vm} \lambda_1}
			\sum_{0\leq \vell\leq \vm}
			\brac{\vm \atop \vell} 
			\norm{   \partial^{\vm-\vell}  u_1 }{V}
			\norm{   \partial^{\vell}  c }{\infty} 
			\bigg) \norm{ u_1}{V}.
		\end{split}
	}
	\end{equation}
\end{lemma}
\begin{proof}
The Leibniz general product rule for the $\vnu$-th derivative of \eqref{variational transform} with respect to $ \vy$ implies
				\begin{equation}\label{nu derivative}
					{\small
				\begin{split}			
					\int_{ D} &\chi_1^{-1} 
					\partial^{\vnu}(  a  \nabla u_1) \cdot  \nabla  v 
					+
					\partial^{\vnu}( b   u_1)   v  - \lambda_1
					\partial^{\vnu}( c   u_1)   v
					=
					\sum_{\vzero\leq\vm< \vnu}\brac{\vnu \atop \vm}
					\partial^{\vnu-\vm}  \lambda_1 
					\int_{ D}   \,  \partial^{\vm}(  c u_1)   v,
				\end{split}}
			\end{equation}
that is valid for all $v\in V$. Notice that the term with $\vm = \vnu$ is moved to the  left-hand side.
	To obtain a bound on the derivatives of the eigenvalue, we specify $ v= u_1$ and split $ \partial^{\vnu}  \lambda_1$ from the remaining terms. This implies
		\begin{equation*}\label{nu derivative2}{
			\scriptsize
		\begin{split}			
			\int_{ D} & \chi_1^{-1} 
\partial^{\vnu}(  a \nabla u_1) \cdot  \nabla  u_1 
			+
\partial^{\vnu}( b   u_1)   u_1  - \lambda_1
\partial^{\vnu}( c   u_1)   u_1
= \partial^{\vnu} \lambda_1 + 
\sum_{\vzero<\vm< \vnu}\brac{\vnu \atop \vm}
			 \partial^{\vnu-\vm}  \lambda_1 
		\int_{ D}   \,  \partial^{\vm}(  c u_1)   u_1,
		\end{split}}
	\end{equation*}
where we have used the normalization $\cM_{ \vy}( u_1, u_1)=1$ in the right-hand side. Next, we apply the Leibniz general product rule to expand the derivatives and arrive at 
	\begin{align}\label{equ: d vnu lambda}
\small	\partial^{\vnu}  \lambda_1
=& \small
			\sum_{\vzero<\vm\leq \vnu}\brac{\vnu \atop \vm}
			\int_{ D} \bigg(\chi_1^{-1} \partial^{\vm}  a\,  \partial^{\vnu-\vm}  \nabla  u_1 \cdot  \nabla  u_1 
			+
			 \partial^{\vm}  b\,  \partial^{\vnu-\vm}   u_1   u_1
			- \lambda_1 \partial^{\vm}  c\,  \partial^{\vnu-\vm}   u_1   u_1\bigg)\notag\\
			&\small\quad-
			\sum_{\vzero<\vm < \vnu}
			\brac{\vnu \atop \vm}  \partial^{\vnu-\vm}  \lambda_1 
			\sum_{0\leq \vell\leq \vm}
			\brac{\vm \atop \vell} 
			\int_{ D} \,  \partial^{\vell}  c\, \partial^{\vm-\vell}  u_1   u_1.  
\end{align}
Observe that the terms with $\vm = \vzero$ vanish in the first sum. Indeed, if $\partial^{\vnu} u_1 \in V$, equation \eqref{variational transform} implies $\cA_{ \vy}( u_1, \partial^{\vnu}  u_1)- \lambda_1 \cM_{ \vy}( u_1, \partial^{\vnu}  u_1)=0$. The statement of the Lemma follows now by the triangle inequality, the Cauchy-Schwarz inequality and the bound $\norm{ u_1}{L^2( D)}\leq \norm{ u_1 }{ V}$, cf. \eqref{simpler Poincare iequ}.
\end{proof}	
 {The existence of the partial derivatives $\partial^{\vnu}\lambda_1$ and $\partial^{\vnu} u_1$ at some fixed $\vy \in U$ for all $\vnu \in \cF$ follows by induction. For the inductive step we will prove that  
$\partial^{\vnu}\lambda_1$ and $\partial^{\vnu}u_1$ are determined by their lower order derivatives. Relation \eqref{equ: d vnu lambda} shows this property for $\partial^{\vnu}\lambda_1$. A similar representation for $\partial^{\vnu}u_1$ will follow from the representations of its orthogonal projection $p_1$ onto the first eigenspace $E(\lambda_1(\vy))$ and the orthogonal complement $p^\bot \in E(\lambda_1(\vy))^\bot$. In this case 
}
	\begin{equation}\label{dnuu1-decomp}
	\partial^{\vnu}u_1
	=
	p_1
	+
	p^\bot, \qquad 
	p_1 \in E(\lambda_1(\vy)), \quad
	p^\bot \in E(\lambda_1(\vy))^\bot
	\end{equation}
{is a uniquely determined bi-orthogonal decomposition for all fixed $\vy \in U$, where} the bi-orthogonality means that the relations $\cA_{\vy}(p_1,p^\bot) = 0$~and $\cM_{\vy}(p_1,p^\bot) = 0$ are simultaneously satisfied. 
{
The required characterizations for $p_1$ and $p^\bot$ is a by-product of the forthcoming Lemma \ref{upper bound for eigenfunction 1st} and Lemma \ref{upper bound for eigenfunction 2nd}. Indeed, the existence and uniqueness of $p^\bot$ is guaranteed by the Lax-Milgram Theorem and the variational formulation \eqref{nu derivative-6}. The component $p_1$ is a multiple of $u_1$ and the proportionality constant will be characterized in \eqref{M vnu 2}.} The triangle inequality reveals
	\begin{equation}\label{dnuu1-norm-decomp}
	\|\partial^{\vnu} u_1\|_V
	\leq
	\|p_1\|_V
	+
	\|p^\bot\|_V.
	\end{equation}
	In the following we will estimate the norms of the projections $p_1$ and $p^\bot$ independently. The following Lemma shows that the norm of $p^\bot$ admits a similar upper bound as the eigenvalue in Lemma \ref{lem:Eval-bnd} that involves lower order derivatives of $\lambda_1$ and $u_1$ only.
\begin{lemma}\label{upper bound for eigenfunction 1st}
	For sufficiently regular solutions of \eqref{variational transform} and $\vnu \in \cF$ with $|\vnu| \geq 1$ there holds
	\begin{equation} \label{v tilder upper}
		\begin{split}
		C_\mu \|p^\bot\|_V 
		&\leq
		\sum_{\vzero < \vm \leq \vnu}\brac{\vnu \atop \vm} 
		\norm{   \partial^{\vnu-\vm}  u_1 }{ V}  
		\brac{
			\norm{   \partial^{\vm}  a }{\infty} 
			+
			\norm{   \partial^{\vm}  b }{\infty} 
			+
			\lambda_1\norm{   \partial^{\vm}  c }{\infty} }\\
		&\quad +
		\sum_{\vzero< \vm < \vnu}\brac{\vnu \atop \vm} 
		\abs{   \partial^{\vnu-\vm}  \lambda_1 }
		\sum_{0\leq \vell\leq \vm}
		\brac{\vm \atop \vell} 
		\norm{   \partial^{\vm-\vell}  u_1 }{ V}  
		\norm{   \partial^{\vell}  c }{\infty} .
		\end{split}
	\end{equation}
\end{lemma}
\begin{proof}
Recall identity \eqref{nu derivative} that is valid for all $v \in V$. In the particular case $v^\bot  \in E(\lambda_1(\vy))^\bot$ it reads
\begin{equation*}\small
\int_{ D} \chi_1^{-1} 
\partial^{\vnu}(  a  \nabla u_1) \cdot  \nabla  v^\bot
			+
\partial^{\vnu}( b   u_1)   v^\bot - \lambda_1
\partial^{\vnu}( c   u_1)   v^\bot
=
\sum_{\vzero<\vm< \vnu}\brac{\vnu \atop \vm}
			 \partial^{\vnu-\vm}  \lambda_1 
		\int_{ D}   \,  \partial^{\vm}(  c u_1)  v^\bot.
\end{equation*}
Notice that the term with $\vm = \vzero$ on the right-hand side has vanished, since $u_1$ and {$v^\bot$} are $\cM_{\vy}$-orthogonal for all fixed $\vy \in U$. Therefore the highest order derivative $\partial^{\vnu} \lambda_1$ is not present in this equation. Next, we apply the Leibniz general product rule and obtain
	\begin{align}
\sum_{\vzero \leq \vm \leq \vnu} 
		\brac{\vnu \atop \vm}			
			\int_{ D} \chi_1^{-1} 
  \partial^{\vm}a  \, \partial^{\vnu - \vm}(\nabla u_1) \cdot  \nabla  v^\bot 
			+
\partial^{\vm} b \, \partial^{\vnu - \vm} u_1 \, v^\bot  - \lambda_1
\partial^{\vm} c \, \partial^{\vnu - \vm} u_1 \, v^\bot &\notag\\
=
\sum_{\vzero<\vm< \vnu}\brac{\vnu \atop \vm}
			 \partial^{\vnu-\vm}  \lambda_1 
			 \sum_{\vzero\leq \vell\leq \vm}\brac{\vm \atop \vell}
		\int_{ D}   \,  \partial^{\vell}  c\, \partial^{\vm - \vell} u_1 \,   v^\bot&. \label{nu derivative4}
\end{align}
	 {We recall the definition of $\wtd \cA_{\vy}$ and the bi-orthogonal decomposition \eqref{dnuu1-decomp} to derive 
\begin{align*}
\wtd \cA_{ \vy}( \partial^{\vnu } u_1,v^\bot)
=
\cA_{ \vy}( p_1,v^\bot)
-
\lambda_1\cM_{ \vy}( p_1,v^\bot)
+
\wtd \cA_{ \vy}( p^\bot,v^\bot)
=
\wtd \cA_{ \vy}( p^\bot,v^\bot).
\end{align*}
Observe that the highest order derivatives $\partial^{\vnu} u_1$ do only appear in the term with $\vm = \vzero$ in the left-hand side of \eqref{nu derivative4}. We keep this term in the left-hand side and put the others to right-hand side to obtain
	\begin{equation}\label{nu derivative-6}
\begin{split}	
&	\small
\wtd \cA_{ \vy}(p^\bot,v^\bot) =
\sum_{\vzero<\vm< \vnu}\brac{\vnu \atop \vm}
			 \partial^{\vnu-\vm}  \lambda_1 
			 \sum_{\vzero\leq \vell\leq \vm}\brac{\vm \atop \vell}
		\int_{ D}   \,  \partial^{\vell}  c\, \partial^{\vm - \vell} u_1 \,   v^\bot \\
&\small-
		\sum_{\vzero < \vm \leq \vnu} 
		\brac{\vnu \atop \vm}			
			\int_{ D} \chi_1^{-1} 
  \partial^{\vm}a  \, \partial^{\vnu - \vm}(\nabla u_1) \cdot  \nabla  v^\bot 
			+
\partial^{\vm} b \, \partial^{\vnu - \vm} u_1 \, v^\bot  - \lambda_1
\partial^{\vm} c \, \partial^{\vnu - \vm} u_1 \, v^\bot 
. 
	\end{split}
	\end{equation}
We now put $v^\bot=p^\bot$ and utilize Lemma \ref{lem:coercive-type} to obtain $\wtd \cA_{ \vy}(p^\bot,p^\bot)\geq  C_\mu \|p^\bot\|_V^2$. }
The statement of the Lemma follows now from \eqref{nu derivative-6} by the triangle inequality, the Cauchy-Schwarz inequality and the Poincar\'e inequality \eqref{simpler Poincare iequ}.
\end{proof}

\begin{lemma}\label{upper bound for eigenfunction 2nd}
	For sufficiently regular solutions of \eqref{variational transform} and $\vnu \in \cF$ with $|\vnu| \geq 1$ there holds
	\begin{equation} \label{w_1 upper}
		\begin{split}
			\norm{p_1}{V}
			&\leq		
			\frac{\|u_1\|_V^2}{2}
			\sum_{\vzero < \vm \leq \vnu}
			\brac{\vnu \atop \vm}
			\norm{\partial^{\vnu-\vm}  u_1}{V}
			\norm{\partial^{\vm} c}{\infty} \\
			&\quad+
			\frac{\|u_1\|_V}{2}
			\sum_{\vzero < \vm < \vnu}
			\brac{\vnu \atop \vm}
			\norm{\partial^{\vnu-\vm}  u_1}{V}
			\sum_{\vzero \leq  \vell\leq \vm} \brac{\vm \atop \vell}
			\norm{\partial^{\vm-\vell} u_1}{V}
			\norm{\partial^{\vell} c}{\infty}.
		\end{split}
	\end{equation}
\end{lemma}
\begin{proof}
Since $p_1$ is the projection of $\partial^{\vnu} u_1$ onto $E(\lambda_1(\vy))$, we have the representation $p_1=\cM_{\vy}(\partial^{\vnu}u_1,u_1) \,u_1$ and hence
\begin{equation}\label{p_1 def & bound}
\norm{p_1}{V}
\leq
\abs{\cM_{\vy}(\partial^{\vnu}u_1,u_1)} 
\, \|u_1\|_V.
\end{equation}
It remains to estimate the modulus of the coefficient.
Recall that \mbox{$\cM_{\vy}(u_1,u_1)=1$}. Thus, taking the $\vnu$-th derivatives of this identity we get by the Leibniz general product rule
	\begin{equation*}\label{M vnu}\small
	0 = \partial^{\vnu} \cM_{\vy}(u_1,u_1)
	 = \sum_{\vzero < \vm < \vnu} 
	\brac{\vnu \atop \vm} \int_D \partial^{\vnu - \vm} u_1 \, \partial^{\vm}(c \, u_1) + 
	\int_D \, (\partial^{\vnu} u_1) \, c \, u_1
	+ \int_D u_1 \, \partial^{\vnu} (c \, u_1).
	\end{equation*}
The highest order derivative $\partial^{\vnu} u_1$ does only appear in the last two integrals. Splitting it from the remaining terms implies
\begin{equation}\label{M vnu 2}\small
2 \cM_{\vy} (\partial^{\vnu}u_1,u_1) = -
\sum_{\vzero < \vm < \vnu} 
	\brac{\vnu \atop \vm} \int_D \partial^{\vnu - \vm} u_1 \, \partial^{\vm}(c \, u_1)  
	- 
	\sum_{\vzero < \vm \leq \vnu} 
	\brac{\vnu \atop \vm} 
	\int_D u_1 \, \partial^{\vm} c \, \partial^{\vnu-\vm} u_1 
\end{equation}
since $\cM_{\vy}$ is symmetric. The statement of the Lemma follows now from \eqref{p_1 def & bound} and \eqref{M vnu 2} by the triangle inequality, the Cauchy-Schwarz inequality and the Poincar\'e inequality \eqref{simpler Poincare iequ}.
\end{proof}

Inequality \eqref{dnuu1-norm-decomp}, Lemma \ref{upper bound for eigenfunction 1st} and Lemma \ref{upper bound for eigenfunction 2nd} imply the estimate
	\begin{equation}
	\begin{split} \label{dnuu1-final}
		&\small
			C_\mu	\|\partial^{\vnu} u_1\|_V \\
		& \small \leq	
		\sum_{\vzero < \vm \leq \vnu}\brac{\vnu \atop \vm} 
		\norm{   \partial^{\vnu-\vm}  u_1 }{ V}  
		\brac{
	 \norm{\partial^{\vm}  a }{\infty} 
			+
	 \norm{  \partial^{\vm}  b }{\infty} 
			+
			( \lambda_1 + \tfrac{C_\mu}{2}\|u_1\|_V^2)\norm{   \partial^{\vm}  c }{\infty} } \\
		 & \small +
		\sum_{\vzero< \vm < \vnu}\brac{\vnu \atop \vm} 
		{\brac{\abs{   \partial^{\vnu-\vm}  \lambda_1 }
		+ \tfrac{C_\mu}{2} \|u_1\|_V \|\partial^{\vnu-\vm}u_1\|_V}}
		\sum_{0\leq \vell\leq \vm}
		\brac{\vm \atop \vell} 
		\norm{   \partial^{\vm-\vell}  u_1 }{ V}  
		\norm{   \partial^{\vell}  c }{\infty} . 
	\end{split}
	\end{equation}

\begin{proof}[Proof of Theorem \ref{gevrey regularity for eigenpairs for ref domain}]
	
We now prove the main theorem by induction. {For the first derivative}, we establish the base for first derivative of the eigenvalue.  Letting $\vnu$ be the unit multiindex, that is $\vnu = \ve$ with $|\ve| = 1$ we utilize \eqref{eigenvalue upper} together with \eqref{lambda_k bound}, \eqref{u_k bound}, \eqref{abc tilder assumption} and the identities \eqref{rel-contrast} to obtain
\begin{align*}
	\abs{\partial^{\ve} \lambda_1}
	\leq
	\overline{u_1}^2\brac{ \frac{\overline{a}\gdnota{1}}{\vR^{\ve}}
	+ \frac{\overline{b}\gdnota{1}}{\vR^{\ve}}
	+
	 \frac{\overline{ \lambda_1}\,
	 \overline{c}\gdnota{1}}{\vR^{\ve}}
    }
	= 	\frac{\overline{\lambda_1}}{\underline{a}}\brac{ \overline{a} 
	+ \overline{b} 
	+
	\overline{ \lambda_1}\,
	 \overline{c} 
    }
	\frac{\gdnota{1}}{\vR^{\ve}}
	=
	\overline{ \lambda_1}
	\frac{\sigma_1\gdnota{1}}{\vR^{\ve}}
\end{align*}
For the eigenfunction we substitute $\vnu=\ve$ into \eqref{dnuu1-final} and use \eqref{lambda_k bound}, \eqref{u_k bound}, \eqref{abc tilder assumption} to get
\begin{align*}
	\footnotesize
	\norm{\partial^{\ve} u_1}{V}
	\leq
		\frac{\overline{u_1}}{C_{\mu}}
		\brac{\frac{\overline{a}\gdnota{1}}{\vR^{\ve}}
		+
		\frac{\overline{b}\gdnota{1}}{\vR^{\ve}}
		+
		\brac{
		\overline{ \lambda_1}
		+
		\frac{C_\mu \overline{u_1}^2}{2}}
		\frac{\overline{c}\gdnota{1}}{\vR^{\ve}}
	}
=
\overline{u_1}
\brac{
	\frac{\overline{a}
	+
	\overline{b}
	+
	\overline{ \lambda_1}\,
	\overline{c}
}{\underline{a} \, \mu}
+
\frac{\overline{u_1}^2 \overline{c}}{2}
}\frac{\gdnota{1}}{\vR^{\ve}}.
\end{align*}
By \eqref{rel-contrast} the term in parentheses equals to $\sigma$.
This concludes the base of the induction.
For the inductive step for the eigenvalue derivative bound, suppose that $\abs{\vnu}\geq 2$ and the bounds \eqref{lambda bound ref} and \eqref{u V bound ref} hold for all multi-indicies of order strictly less than $\vnu$. In particular, {the inductive assumption} imply
\begin{align}\label{mib-1} \small
	\max \left(\frac{\abs{ \partial^{\vnu-\vm} \lambda_1( \vy)}}{\overline{ \lambda_1}}, \frac{\norm{\partial^{\vnu-\vm} u_1}{V}}{\overline{u_1} } 
	\right)
	\leq
	\frac{{\sigma}\rho^{\abs{\vnu - \vm}} \gdnota{\abs{\vnu-\vm}}}{\rho \vR^{{\vnu-\vm}}(\abs{\vnu-\vm}!)^{1-\delta}}, 
	\qquad \vzero < \vm < \vnu, \quad |\vnu| \geq 2.
\end{align}
Since $\sigma, \rho \geq 1$ we have for $\vm > \vzero$ the trivial upper bound {$\rho^{|\vnu - \vm|} \leq  \rho^{|\vnu|-1}$}, and thereby 
\begin{align}\label{mib-2}
	\norm{\partial^{\vnu-\vm} u_1}{V}
	\leq
	\frac{\overline{u_1}{\sigma}\rho^{\abs{\vnu}} \gdnota{\abs{\vnu-\vm}}}{\rho^2 \vR^{{\vnu-\vm}}(\abs{\vnu-\vm}!)^{1-\delta}}, 
	\qquad \vzero < \vm \leq \vnu, \quad |\vnu| \geq 2.
\end{align}
Notice that this estimate indeed holds for the extended range including $\vm = \vnu$, since in this case, according to \eqref{u_k bound} {and $\sigma, \rho \geq 1$}, we have $\|u_1\|_V \leq \overline{u_1} \leq \overline{u_1}  {\sigma} \rho^{|\vnu|-2}$.
The following estimate also follows with similar arguments
\begin{align}\label{mib-3}
	\norm{\partial^{\vm-\vell} u_1}{V}
	\leq
	\frac{\overline{u_1}{\sigma} \rho^{\abs{\vm}} \gdnota{\abs{\vm-\vell}}}{\rho \vR^{{\vm-\vell}}(\abs{\vm-\vell}!)^{1-\delta}}, \qquad \vzero\leq \vell \leq \vm, \quad \vzero < \vm < \vnu, \quad |\vnu| \geq 2.
\end{align}
We are now ready to proceed with the inductive step. For the derivative of the eigenvalue recall \eqref{eigenvalue upper}, the regularity assumptions \eqref{abc tilder assumption} and the bounds \eqref{mib-1} -- \eqref{mib-3} to obtain 
\begin{align*}
	&\small
		\abs{ \partial^{\vnu}  \lambda_1}
		  \leq
		\overline{ u_1}
		\sum_{\vzero< \vm \leq \vnu}\brac{\vnu \atop \vm} 
		\frac{\overline{u_1}{\sigma}\rho^{\abs{\vnu}} \gdnota{\abs{\vnu-\vm}}}{\rho^2 \vR^{{\vnu-\vm}}(\abs{\vnu-\vm}!)^{1-\delta}}  
		\brac{
			 \frac{(\overline{a}+\overline{b})\gdnota{\abs{\vm}} }{\vR^{\vm}(\abs{\vm}!)^{1-\delta}}
			+ 
			\overline{ \lambda_1}
			\frac{\overline{c}\gdnota{\abs{\vm}} }{\vR^{\vm}(\abs{\vm}!)^{1-\delta}}}\\
		&\small
		+  \overline{ u_1}
		\sum_{\vzero<\vm < \vnu}
		\brac{\vnu \atop \vm}
		\frac{ \overline{ \lambda_1} {\sigma} \rho^{\abs{\vnu-\vm}}\gdnota{\abs{\vnu-\vm}}}{\rho \vR^{{\vnu-\vm}}(\abs{\vnu-\vm}!)^{1-\delta}}  
		\sum_{0\leq \vell\leq \vm}
		\brac{\vm \atop \vell} 
		\frac{\overline{u_1}{\sigma} \rho^{\abs{\vm}} \gdnota{\abs{\vm-\vell}}}{\rho \vR^{{\vm-\vell}}(\abs{\vm-\vell}!)^{1-\delta}}
		 \frac{\overline{c}\gdnota{\abs{\vell}} }{\vR^{\vell}(\abs{\vell}!)^{1-\delta}},
\end{align*} 
Since $\delta \geq 1$ and $ |\vnu - \vm|!|\vm|! \leq |\vnu|!$ and likewise $|\vnu - \vm|! |\vm - \vell|!  |\vell|!\leq |\vnu|!$ the term $(|\vnu|!)^{\delta-1}$ can be factored out by making the right-hand side larger. Collecting the remaining terms we obtain
\begin{align*}
	\abs{ \partial^{\vnu}  \lambda_1}
	&\leq
	\overline{ u_1}^2
	(\overline{a}+\overline{b}+
	\overline{\lambda_1}\overline{c})
	\frac{{\sigma} \rho^{\abs{\vnu}}}{\rho^2\vR^{\vnu}}
	\frac{1}{(\abs{\vnu}!)^{1-\delta}}
	\sum_{\vzero< \vm \leq \vnu}\brac{\vnu \atop \vm} 
	\gdnota{\abs{\vnu-\vm}}
	\gdnota{\abs{\vm}} \\
	&
	+  \overline{ \lambda_1} \overline{ u_1}^2
		\, \overline{c}
			\frac{{\sigma^2}\rho^{\abs{\vnu}}}{\rho^2 \vR^{\vnu}}
	\frac{1}{(\abs{\vnu}!)^{1-\delta}}
	\sum_{\vzero<\vm < \vnu}
	\sum_{0\leq \vell\leq \vm}
	\brac{\vnu \atop \vm}
	\brac{\vm \atop \vell} 
	\gdnota{\abs{\vnu-\vm}}  
	\gdnota{\abs{\vm-\vell}}
	\gdnota{\abs{\vell}}  .
\end{align*}
{Recall that $|\vnu| \geq 2$.} {Thanks to} the bounds \eqref{multiindex-est-3}, \eqref{multiindex-est-4} from the appendix, the first sum equals to $3 [\tfrac{1}{2}]_{|\vnu|}$, the second is bounded by $8 [\tfrac{1}{2}]_{|\vnu|}$. Notice moreover that $ \overline{u_1}^2=\underline{a}^{-1}\overline{\lambda_1}$ as we arrive at
\begin{align}\label{lambda vnu induction}
	\abs{ \partial^{\vnu}  \lambda_1}
	\leq 
	\overline{ \lambda_1}{\sigma}
	\brac{
	\frac{
	3 (\overline{a}+\overline{b}
	+\overline{\lambda_1}\overline{c})}
	{\underline{ a}}
	+8 \overline{ u_1}^2\overline{c}\,{\sigma}}
	\frac{\rho^{\abs{\vnu}}}{\rho^2\vR^{\vnu}}
	\frac{\gdnota{\abs{\vnu}}}{(\abs{\vnu}!)^{1-\delta}}.
\end{align}
According to the definition \eqref{def-contrast} of the coefficient contrast $\con_a$ and $\con_c$ the term in parentheses equals $\rho_1$. Since $\rho_1 \leq \rho$ this implies the statement \eqref{lambda bound ref} for the derivatives of the eigenvalue.

The inductive step for the eigenfunction is similar. Here we combine \eqref{dnuu1-final} with the regularity assumptions \eqref{abc tilder assumption} and the bounds \eqref{mib-1} -- \eqref{mib-3} to obtain 
\begin{align*}
		&\small
		\|\partial^{\vnu} u_1\|_V  \\
		&\leq
		C_{\mu}^{-1} 
		\sum_{\vzero < \vm \leq \vnu}\brac{\vnu \atop \vm} 
		\frac{\overline{u_1}{\sigma} \rho^{\abs{\vnu}}\gdnota{\abs{\vnu-\vm}}}{\rho^{2}\vR^{{\vnu-\vm}}(\abs{\vnu-\vm}!)^{1-\delta}}  
		\frac{(\overline{a}+\overline{b}+\brac{\overline{ \lambda_1}
			+ \frac{1}{2}C_{\mu}\overline{u_1}^2}\overline{c})\gdnota{\abs{\vm}} }{\vR^{\vm}(\abs{\vm}!)^{1-\delta}}\\
		&
		\small
		+
C_{\mu}^{-1}
		\sum_{\vzero< \vm < \vnu}\brac{\vnu \atop \vm} 
		\frac{ \overline{ \lambda_1}{\sigma} \rho^{\abs{\vnu-\vm}}\gdnota{\abs{\vnu-\vm}}}{\rho\vR^{{\vnu-\vm}}(\abs{\vnu-\vm}!)^{1-\delta}}  
		\sum_{0\leq \vell\leq \vm}
		\brac{\vm \atop \vell} 
		\frac{\overline{u_1}{\sigma} \rho^{\abs{\vm}}\gdnota{\abs{\vm-\vell}}}{\rho\vR^{{\vm-\vell}}(\abs{\vm-\vell}!)^{1-\delta}} 
		\frac{\overline{c}\gdnota{\abs{\vell}} }{\vR^{\vell}(\abs{\vell}!)^{1-\delta}}\\
		&
		\small
		+
		\frac{\overline{u_1}}{2}
		\sum_{\vzero< \vm < \vnu}\brac{\vnu \atop \vm} 
		\frac{\overline{u_1}{\sigma} \rho^{\abs{\vnu-\vm}}\gdnota{\abs{\vnu-\vm}}}{\rho\vR^{{\vnu-\vm}}(\abs{\vnu-\vm}!)^{1-\delta}}  
		\sum_{0\leq \vell\leq \vm}
		\brac{\vm \atop \vell} 
		\frac{\overline{u_1}{\sigma} \rho^{\abs{\vm}}\gdnota{\abs{\vm-\vell}}}{\rho\vR^{{\vm-\vell}}(\abs{\vm-\vell}!)^{1-\delta}} 
		\frac{\overline{c}\gdnota{\abs{\vell}} }{\vR^{\vell}(\abs{\vell}!)^{1-\delta}}.
\end{align*}
Extracting the term with the factorial and recalling that $C_\mu = \underline{a} \,\mu$ we find
\begin{equation*} 
	\begin{split}
		&\small\norm{   \partial^{\vnu}  u_1 }{ V}  \\
		& \small \leq
		\overline{u_1}
		\brac{
		\frac{\overline{a}+\overline{b} +\overline{ \lambda_1}\, \overline{c}}{\underline{a}\, \mu}
		+\frac{\overline{u_1}^2 \overline{c}}{2}
		}
		 \frac{{\sigma}\rho^{\abs{\vnu}}}{\rho^2 \vR^{\vnu}}
		 \frac{1}{(\abs{\vnu}!)^{1-\delta}}
		\sum_{\vzero < \vm \leq \vnu}\brac{\vnu \atop \vm} 
		\gdnota{\abs{\vnu-\vm}}  
		\gdnota{\abs{\vm}}\\
		&\small +
		\overline{u_1}
		\brac{
		\frac{\overline{ \lambda_1} \overline{c}\,{\sigma}}{\underline{a}\, \mu}
		+
		\frac{\overline{u_1}^2 \overline{c}\,{\sigma}}{2}
		}
		\frac{{\sigma} \rho^{\abs{\vnu}}}{\rho^2 \vR^{\vnu}}		
		 \frac{1}{(\abs{\vnu}!)^{1-\delta}}
		\sum_{\vzero< \vm < \vnu}
		\sum_{0\leq \vell\leq \vm}
		\brac{\vnu \atop \vm} 
		\brac{\vm \atop \vell} 
		\gdnota{\abs{\vnu-\vm}}  
		\gdnota{\abs{\vm-\vell}} 
		\gdnota{\abs{\vell}} 
	\end{split}
\end{equation*}
By Lemma \ref{appendix lem} the sums above can be estimated by $3 [\tfrac{1}{2}]_{|\vnu|}$ and $8 [\tfrac{1}{2}]_{|\vnu|}$, respectively, and hence 
\begin{equation*} 
	\begin{split}
		\norm{   \partial^{\vnu}  u_1 }{ V}  
		\leq
		\overline{u_1}
		\brac{\frac{3(\overline{a}+\overline{b})+{(3+8\sigma)}\overline{ \lambda_1}\, \overline{c} }{\underline{a} \, \mu}
		+
		{(3+8\sigma)}\frac{\overline{u_1}^2\,\overline{c}}{2}		
		}
		\frac{\sigma\rho^{\abs{\vnu}}}{\rho^2 \vR^{\vnu}}
	\frac{\gdnota{\abs{\vnu}}}{(\abs{\vnu}!)^{1-\delta}}
\end{split}
\end{equation*} 
The definition \eqref{def-contrast} of the coefficient contrast $\con_a$ and $\con_c$ implies that the term in the parentheses coincides with {$\rho$. This implies the statement \eqref{u V bound ref} and, thus,} finishes the proof. 
\end{proof}
\section{Applications and numerical experiments}\label{sec: num exp}
In this section we demonstrate how the regularity analysis from the previous sections can be used to rigorously justify the convergence of parametric integration in two particular applications: \linebreak i)~one-dimensional parametric integration with the Gauss-Legendre quadrature and ii)~high-dimensional integration with the Quasi-Monte Carlo method. The new contributions here are twofold. First, we compare the convergence behaviour for the analytic ($\delta = 1$) and Gevrey non-analytic ($\delta > 1$) cases. Second, our analysis of the QMC method extends the existing results from \cite{Gilbert2019,GilbertScheichl2} and demonstrate their consistency with the analysis of the source problem in \cite{KuoSchwabSloan2012}, see Remark \ref{QMC remark} below.

In what follows we focus on the analysis of the smallest eigenvalue $\lambda_1$. Notice, however, that if the assumptions of Theorem \ref{gevrey regularity for eigenpairs for general equ} are satisfied, the results immediately extend to $\cG(u_1)$, where $\cG \in V^*=H^{-1}(D)$ is a linear functional, since
\[
|\partial^{\vnu} \cG(u_1)| \leq \|\cG\|_{V^*} \|\partial^{\vnu} u_1\|_{V}
\]
and, thus, $\cG(u_1)$ satisfies the same type of estimate as \ref{lambda bound} for $\lambda_1$. 

\subsection{Gauss-Legendre quadrature} \label{sec: Gauss-Legendre Quadrature}
Let us fix the domain $D= (0,1)^2$ and consider the problem \eqref{gen equation} with $b\equiv 0$ and $c\equiv 1$ and $a$ is one of following functions 
\begin{equation}\label{def-a1}
a^{(1)} (\vx,y) = 2+\sin\brac{\pi (x_1+x_2+y)},
\end{equation}
\begin{equation}\label{def-a2}
a^{(2)} (\vx,y) = 1+\brac{x_1+x_2}\exp\brac{-\frac{1}{\sqrt{y+1}}}.
\end{equation}
Here $y$ is a scalar real random variable uniformly distributed in $[-1,1]$. The aim is to compute the (rescaled) expected value of the smallest eigenvalue
\begin{equation}\label{Elambda}
I(\lambda_1) =  \int_{-1}^1 \lambda_1(y) \, dy.
\end{equation}
Since $\lambda_1$ is not available in closed form, the integral \eqref{Elambda} can be approximated by numerical quadrature, e.g. the Gauss-Legendre quadrature, which will be efficient for the case of a single real-valued parameter. Let $\{\xi_i,w_i\}_{i=1}^n$ be the nodes and weights of the $n$-point Gauss-Legendre quadrature
\begin{equation}
Q_n[\lambda_1] =  \sum_{i=1}^n w_i \lambda_1(\xi_i).
\end{equation}
We are interested in the behaviour of the quadrature error
\begin{equation}
\varepsilon_n = |I(\lambda_1) - Q_n[\lambda_1]|
\end{equation}
with increasing $n$.  It is known that the convergence of $\varepsilon_n$ is strongly related to the regularity of $\lambda_1$ with respect to $y$. In particular \cite[Theorem 5.2]{ChernovSchwab2012} implies that
\begin{equation}
\varepsilon_n \leq C \exp( - r n^{1/\delta}).
\end{equation}
with positive constants $C$ and $r$ independent of $n$.
\begin{itemize}
\item[(1)]
In the case of the analytic diffusion coefficient $a^{(1)}$ as in \eqref{def-a1}, Theorem \ref{gevrey regularity for eigenpairs for general equ} implies that $\lambda_1$ is analytic in $y$, i.e. $\delta = 1$, and therefore we expect
\begin{equation}\label{eps1}
\varepsilon_n^{(1)} \leq C \exp (-rn).
\end{equation}
\item[(2)] The diffusion coefficient $a^{(2)}$ is not analytic near $y = -1$, but is Gevrey-$\delta$ uniformly for all $y \in [-1,1]$ with $\delta \geq 3$, see \cite{ChenRodino1996} and \cite[Section 6]{ChernovSchwab2012}. Theorem \ref{gevrey regularity for eigenpairs for general equ} implies that $\lambda_1$ is Gevrey-$\delta$ with the same $\delta = 3$ and hence we expect
\begin{equation}\label{eps2}
\varepsilon_n^{(2)} \leq C \exp (-rn^{1/3}).
\end{equation}
\end{itemize}
In order to observe the behaviour predicted in \eqref{eps1} and \eqref{eps2}, we solve deterministic equations~\eqref{variational transform} in every quadrature point $y = \xi_i$ numerically by the piecewise linear finite element method on a very fine uniform triangular mesh having $16.129$ degrees of freedom. Since $I(\lambda_1)$ is not available in closed form, we approximate it by the very fine Gauss-Legendre quadrature $Q_{n^*}[\lambda_1]$ with $n^* = 40$ quadrature nodes for $a^{(1)}$ and $n^* = 123$ quadrature nodes for~$a^{(2)}$. As an eigensolver we use the inverse power iteration with the absolute error tolerance of $10^{-14}$.

In {the left panel of} Figure \ref{fig:GL}, we plot the relative error $\varepsilon_n^{(1)}$ against the number of quadrature points $n$ in the semi-logarithmic scale. The reference line clearly shows the linear trend of the type $-r n + \log C$ and thereby confirms \eqref{eps1}.
 
In {the right panel of Figure \ref{fig:GL}}, we plot the relative error $\varepsilon_n^{(2)}$ with respect to the third root of the number of quadrature points $m := n^{1/3}$ in the semi-logarithmic scale. Here we can also observe the linear trend of the type $-r m + \log C$. This confirms \eqref{eps2} and thereby demonstrates the meaning and validity of Theorem \ref{gevrey regularity for eigenpairs for general equ}.

\begin{figure}[h]
\centerline{\includegraphics[width=1\textwidth]{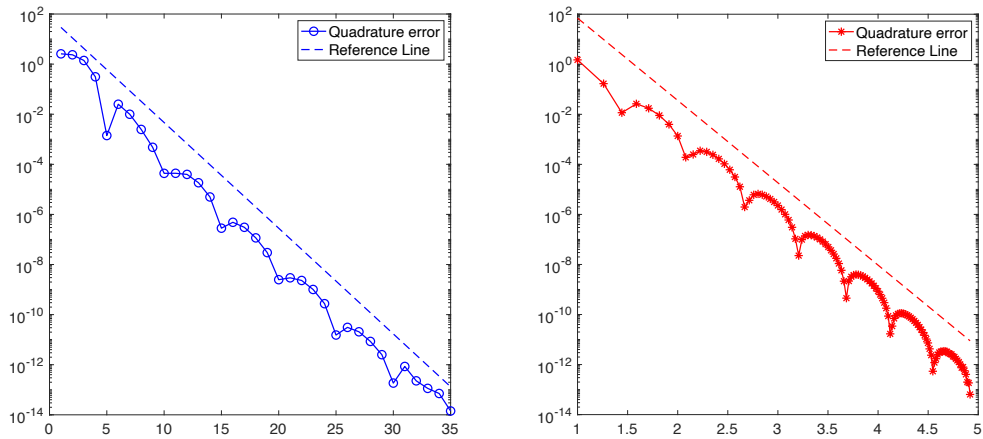}}
	\caption{The quadrature error $\varepsilon^{(1)}_n$ with respect to the number $n$ of quadrature points (left) and the quadrature error $\varepsilon^{(2)}_n$ with respect to $m = n^{1/3}$ (right).}\label{fig:GL}
\end{figure}

\subsection{Quasi-Monte Carlo method for Gevrey functions} \label{sec: QMC}

Let $Y=\left[-\frac{1}{2},  \frac{1}{2}\right]$, for any given $s\in \mN$  we denote by $\vy_s= (y_1,\dots,y_s,0,0,\dots) $ the $s$-dimensional truncation of $\vy \in U= Y^\mN$. For a function $F: Y^\mN \mapsto \mR $, our quantity of interest is the integral of the form
\begin{align}\label{approx integral}
I(F) =	
\int_{Y^\mN} F(\vy) \, d \vy
\end{align}

In this section we apply our main regularity result in Theorem \ref{gevrey regularity for eigenpairs for general equ} to analyse the convergence rate of Quasi-Monte Carlo (QMC) method for $F = \lambda_1$. (As mentioned at the beginning of this section, this result extends immediately to $F=\cG(u_1)$, where $\cG\in V^*:= H^{-1}(D)$ is a linear functional.) 
The QMC approximation reads
\begin{align}\label{QMC quad def}
	Q^{\Delta}_{s,n}(F) := 
	\frac{1}{n} \sum_{i=1}^n F\brac{\left\{ \frac{i \vz_{s}}{n} +\Delta  \right\} -\frac{1}{2}},
\end{align}
which is a randomly shifted lattice rule with the generating vector $\vz_{s} \in \mN^s$. Here~$\Delta$ is a random shift, which is uniformly distributed over the cube $(0,1)^s$ and $n$ is the number of quadrature points. The braces in \eqref{QMC quad def} indicate the fractional part of each component of the argument vector. Notice that $Q^{\Delta}_{s,n}(F)$ depends on the random shift and therefore is a random variable. A popular measure of accuracy is the root mean square error 
\[
{\rm RMSE} = \sqrt{\mathbb{E}|I(F)-Q^{\Delta}_{s,n}(F)|^2},
\]
where $\mathbb E$ is the expectation with respect to the random shifts $\Delta$. Moreover, $Q^{\Delta}_{s,n}(F)$
approximates only in the first $s$ components, therefore it is natural to introduce the truncation 
\begin{align}
I_s(F) =	
\int_{Y^s} F({\vy}_s)\, d y_1 \dots d y_s
\end{align}
and, using the triangle inequality to get the error decomposition
\begin{align} \label{error-deco}
{\rm RMSE}
	&\leq
	\abs{I(F)-I_s(F)}
	+
	\sqrt{\mE\brac{ \abs{I_s(F)-Q^{\Delta}_{s,n}(F)}^2}}.
\end{align}
Assume from now on that the assumptions of Theorem~\ref{gevrey regularity for eigenpairs for general equ} are valid.
The first summand on the right-hand side of \eqref{error-deco}, the truncation error, can only converge to zero, if $F$ becomes ``less dependent'' on $y_s$ as $s \to \infty$. 
A sufficient condition that rigorously implies the desired behaviour is
	\begin{align}\label{beta-lp}
		 \norm{\vbeta}{\ell^p}:= \bigg(\sum_{j = 1}^\infty \beta_j^p\bigg)^{\frac{1}{p}} < \infty
	\end{align}
for some $p \in (0,1]$ and $\beta_j:={\wtd R_j}^{-1}$, i.e. $\wtd R_s$ in Theorem~\ref{gevrey regularity for eigenpairs for general equ} grows sufficiently fast in $s$. Following closely the arguments in \cite[Theorem 4.1]{Gilbert2019}, this implies 
	\begin{align}\label{equ: lambda truncation error}
	\abs{I(F) - I_s(F)}
	\leq
	C_1\,s^{-2\brac{\frac{1}{p}-1}},
\end{align}
where  $C_1$ depends on $\delta \geq 1$, but is independent of $s$. For completeness, we provide the proof in the general case in Lemma \ref{general trunc theorem} in the Appendix. 

The estimate for the second summand in the right-hand side of \eqref{error-deco}, the quadrature error, is slightly more delicate, but can be constructed also for the case $\delta > 1$ following the arguments of \cite[Theorem 4.2]{Gilbert2019} and \cite[Theorem 6.4]{KuoSchwabSloan2012}. The complete proof is given in Lemma \ref{QMC theorem} in the Appendix. As a simple corollary, for a fixed integer $s$ and $n$ being a power of 2, a QMC quadrature rule $Q^{\Delta}_{s,n}$ can be explicitly constructed, such that
\begin{equation}\label{errorQMC}
	\sqrt{\mE\brac{ \abs{I_s(F)-Q^{\Delta}_{s,n}(F)}^2}}
	\leq C_2 n^{-\frac{1}{2\vartheta}},
\end{equation}
where $C_2$ is independent of $n$ and
\begin{align}\label{QMC convergence rate}
		\vartheta = 
		\left\{\begin{matrix}
			\omega &\text{for some } \omega \in (\tfrac{1}{2},1), & \text{when } p\in (0,\tfrac{2}{3\delta}],
			\\
			\frac{\delta p}{2-\delta p},    &  & \text{when } p\in (\tfrac{2}{3\delta},\tfrac{1}{\delta}]. 
		\end{matrix}\right. 
	\end{align}
This result requires assumptions of Theorem \ref{gevrey regularity for eigenpairs for general equ} and \eqref{beta-lp} in the reduced range $p \in (0,\delta^{-1})$. The result is still valid for $p = \delta^{-1}$ if \eqref{beta-lp} is replaced with $\|\vbeta\|_{\ell^p} < \sqrt{6}$. In this case the convergence rate deteriorates to the rate of the plain Monte Carlo estimator, that is
\begin{equation}
	\sqrt{\mE\brac{ \abs{I_s(F)-Q^{\rm MC}_{s,n}(F)}^2}}
	\leq C_3 n^{-\frac{1}{2}}.
\end{equation}
Here
\begin{equation}
Q^{\rm MC}_{s,n}(F) := \frac{1}{n} \sum_{i=1}^n F(\vy_s^{(i)})
\end{equation}
and $\vy_s^{(i)}$ are independent samples from the uniform distribution on $Y^s$. The constant~$C_3$ is determined by the variance of $F(\vy_s)$ and thereby is independent of $n$.

Observe that the Gevrey-$\delta$ non-analytic regularity (i.e., $\delta > 1$) has a more significant effect on the QMC error \eqref{errorQMC} rather than on the truncation error \eqref{equ: lambda truncation error}. For this reason in the forthcoming example we fix the length of the parameter vector to $s = 100$ terms and concentrate specifically on the behaviour of the QMC error. Let $D= (0,1)^2$ and consider the problem \eqref{gen equation} with $b\equiv 0$ and $c\equiv 1$ and $a$ being one of the following functions
\begin{equation}\label{QMC-a1-new}
	a^{(1)} (\vx,\vy) = 
	2+2\exp\brac{-{\zeta(5)}+\sum_{j=1}^{100}
		{j^{-5}}
		 \sin(j\pi  x_1)\sin(j\pi x_2)\,y_j},
\end{equation}
\begin{equation}\label{QMC-a2-new}
	a^{(2)} (\vx,\vy) = 3+\frac{1}{\zeta(5)} \sum_{j=1}^{100} j^{-5} \sin(j\pi  x_1)\sin(j\pi x_2)\exp\brac{
		-\frac{1}{y_j+\tfrac{1}{2}}}
	,
\end{equation}
Here $y_j$ is a scalar real random variable uniformly distributed in $[-\tfrac{1}{2},\tfrac{1}{2}]$ for all $j\in \mN$ and $\zeta(\cdot)$ denotes the Riemann zeta function. Clearly, for all $j\in \mN$, we have that $\vbeta^{(i)}\in \ell^p$ with any $p>\frac{1}{5}$. Moreover, $a^{(1)}$ is analytic in~$\vy$ ($\delta^{(1)} = 1$), whereas $a^{(2)}$ is Gevrey-$\delta$ with $\delta^{(2)} = 2$. From Theorem \ref{gevrey regularity for eigenpairs for general equ} we know that this regularity carries over to the corresponding smallest eigenvalues $\lambda_1^{(1)}$ and $\lambda_1^{(2)}$ with the same~$\delta$. To check this result numerically, we compute the eigenvalues by the piecewise linear finite element method on a very fine uniform triangular mesh having $16.129$ degrees of freedom (i.e. the mesh consists of congruent right isosceles triangles of the diameter $h = \sqrt{2}/128$), so that the effect of the finite element discretization is negligible. As in Section \ref{sec: Gauss-Legendre Quadrature}, we use the inverse power iteration with the absolute error tolerance of $10^{-14}$.
The outer expectation is approximated by the empirical mean of $R = 8$ runs. The $j$-th run corresponds to an independent sample of the random shift vector $\Delta^{(j)}$ drawn from the uniform distribution on the unit cube $(0,1)^s$. By $Q^{(j)}_{s,n}(\lambda_1^{(k)})$ we denote the corresponding QMC quadrature. Then the relative QMC error is given by 
\begin{equation}
	\varepsilon_n^{{\rm QMC},(k)}=\sqrt{
		\frac{1}{R} \sum_{j=1}^R
		 \abs{\frac{ I_s^*(\lambda_1^{(k)})-Q^{(j)}_{s,n}(\lambda_1^{(k)})}{I_s^*(\lambda_1^{(k)})}}^2},
		\end{equation}
and analogously for the plain Monte Carlo approximation $\varepsilon_n^{{\rm MC},(k)}$. In both cases the reference value $I_s^*(\lambda_1^{(k)})$ corresponds to the highest level of the QMC approximation.

Since $p \in (0,\frac{2}{3 \delta^{(2)}}) \subset (0,\frac{2}{3 \delta^{(1)}})$, we expect from the above theory that
$\varepsilon^{{\rm QMC},(1)}_n$ and $\varepsilon^{{\rm QMC},(2)}_n$ are approximately proportional to $n^{-1}$, whereas $\varepsilon^{{\rm MC},(1)}_n$ and $\varepsilon^{{\rm MC},(2)}_n$ are approximately proportional to $n^{-\frac{1}{2}}$. {In Figure \ref{fig:QMC}} we clearly observe that this convergence behaviour is reproduced.

\begin{figure}[h]
\centerline{\includegraphics[width=0.5\textwidth]{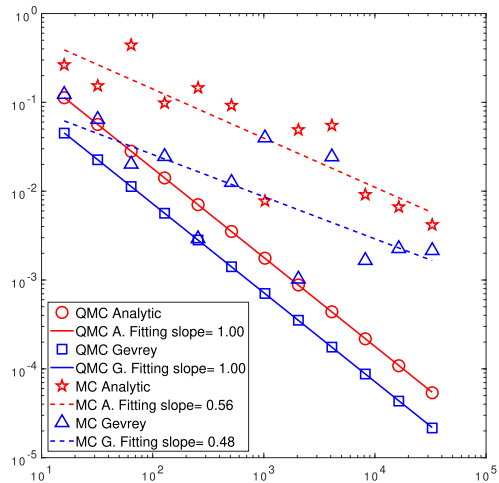}}
	\caption{{Convergence  of the QMC errors $\varepsilon_n^{{\rm QMC},(1)}$ and $\varepsilon_n^{{\rm QMC},(2)}$, and the MC errors $\varepsilon_n^{{\rm MC},(1)}$ and $\varepsilon_n^{{\rm MC},(2)}$ for the analytic and the Gevrey class setting.}}\label{fig:QMC}
\end{figure}

\section{APPENDIX}\label{sec: appendix}
Here we collect some combinatorial results required in the proof of Theorem \ref{gevrey regularity for eigenpairs for ref domain}, as well as novel estimates on the QMC quadrature error for Gevrey class functions needed in Subsection \ref{sec: QMC}.

\begin{lemma} \label{sum binomialcoef}
	For all $\vnu\in \mN^d_0$ and $0\leq r\leq \abs{ \vnu}$, we have
	$$
	\sum_{\vzero \leq \vm \leq \vnu \atop \abs{ \vm}=r}
	\brac{ \vnu \atop  \vm}
	=
	\brac{\abs{ \vnu} \atop r}.
	$$
\end{lemma}
\begin{proof}
	For a fixed $z\in \mR$ and $\vx\in \mR^d$, we define functions $h(\vx),g(\vx)$ and $f(\vx)$ as  
	\begin{align*}
		\small h(\vx)=\exp\brac{z \sum_{i=1}^d x_i},~~ 
		g(\vx)=\exp\brac{\sum_{i=1}^d x_i}, ~~
		f(\vx)=h(\vx)\,g(\vx)=\exp\bigg((z+1) \sum_{i=1}^d x_i\bigg).
	\end{align*}
On the one hand, for every $ \vnu\in \mN^d_0$, taking the $\vnu$-derivative with respect to $\vx$ and evaluating at $\vx=\vzero$, we derive
	\begin{align*}
		\partial^{ \vnu} f(\vzero)
		=
		(z+1)^{\abs{\vnu}} f(\vzero)
		=
		\sum_{r=0}^{\abs{ \vnu}}
		\brac{\abs{\vnu}\atop r}
		z^{r}.
	\end{align*}
	On the other hand, by the Leibniz product rule, we get
	\begin{align*}
		\partial^{ \vnu} f(\vzero)
		=
		\sum_{\vzero\leq \vm\leq \vnu}
		\brac{\vnu \atop \vm}
		\partial^{\vnu-\vm}
		g(\vzero)
		\partial^{\vm}
		h(\vzero)
		=
		\sum_{\vzero\leq \vm\leq \vnu}
		\brac{\vnu \atop \vm}
		z^{\abs{\vm}}
		=
		\sum_{r=0}^{\abs{ \vnu}}
		\sum_{\vzero \leq \vm \leq \vnu \atop \abs{ \vm}=r}
		\brac{\vnu \atop \vm}
		z^{r}.
	\end{align*}
	The proof is concluded by comparing the coefficients in front of $z^r$ in the above equations.
\end{proof}
\begin{lemma}\label{appendix lem}
	For three multiindices $\vnu, \vm, \vell \in \cF$ it holds that
\begin{equation}
		\sum_{0<\vm \leq \vnu}
		\brac{\vnu \atop \vm} 
		[\tfrac{1}{2}]_{|\vnu-\vm|}
		[\tfrac{1}{2}]_{|\vm|}
		\leq 	
		3 [\tfrac{1}{2}]_{|\vnu|}\label{multiindex-est-3},
	\end{equation}
	\begin{equation}
	\sum_{0<\vm< \vnu}
	\sum_{0\leq\vell\leq \vm}
	\brac{\vnu \atop \vm} 
	\brac{\vm \atop \vell} 
	[\tfrac{1}{2}]_{|\vnu-\vm|}
	[\tfrac{1}{2}]_{|\vm-\vell|}
	[\tfrac{1}{2}]_{|\vm|}
	\leq 		
	8 [\tfrac{1}{2}]_{|\vnu|}\label{multiindex-est-4}.
\end{equation}
{Moreover, in \eqref{multiindex-est-3}, the two sides are equal if and only if $\abs{\vnu}\geq 2$.}	
\end{lemma}
\begin{proof}
For the first statement, appying Lemma \ref{sum binomialcoef}  and Lemma \ref{lem:ff-main-id}, we arrive at
\begin{equation*}
\begin{split}
	\sum_{0<\vm \leq \vnu}
	\brac{\vnu \atop \vm } 
	[\tfrac{1}{2}]_{|\vnu-\vm|}
	[\tfrac{1}{2}]_{|\vm|}
	&=
	\sum_{r=1}^{\abs{\vnu}}
	\sum_{\vzero \leq \vm \leq \vnu \atop \abs{ \vm}=r} \brac{\vnu \atop r}
	[\tfrac{1}{2}]_{|\vnu|-r}
	[\tfrac{1}{2}]_{r}\\
	&=
	\sum_{r=1}^{\abs{\vnu}}
	\brac{|\vnu| \atop r} 
	[\tfrac{1}{2}]_{|\vnu|-r}
	[\tfrac{1}{2}]_{r}
	\leq
	3 [\tfrac{1}{2}]_{|\vnu|}.
	\end{split}
\end{equation*}
According to Lemma \ref{lem:ff-main-id}, the inequality sign in this estimate can be replaced by the equality sign for $|\vnu|  \geq 2$ .
The proof of estimate \eqref{multiindex-est-4} is analogous. For this we apply Lemma \ref{lem:ff-main-id} twice to obtain the final result.
\end{proof}
\begin{lemma} \label{general trunc theorem}
	Let $F$ be a smooth mapping $Y^\mN \mapsto \mR$. 
	Suppose that there exist  $\delta \geq 1$, $p\in (0,1]$, a constant $C_F>0$ and a positive sequence $\vbeta$ such that 
	\begin{align}\label{QMC lemma assumption}
		\abs{\partial^{\vnu} F (\vecxi) }
		\leq
		C_F  \, \vbeta^{\vnu} (\abs{\vnu}!)^{\delta} 
		\quad 
		\text{ and }
		\quad
		\norm{\vbeta}{\ell^p}:= \bigg(\sum_{j\geq 1} \beta_j^p\bigg)^{\frac{1}{p}} < \infty
	\end{align}
	for all $\vecxi \in Y^\mN$ and $\vnu \in \cF$. 
	Then the truncation error is bounded by
	\begin{align}\label{equ: truncation error}
		\abs{I(F) - I_s(F)}
		\leq
		C_F \,C_{\delta}\, s^{-2\brac{\frac{1}{p}-1}},
	\end{align}
	where $C_{\delta}= \frac{2^{\delta}}{24}
	\norm{\vbeta}{\ell^p}^2 \min\brac{\frac{p^2}{(1-p)^2},1}$.
\end{lemma}
\begin{proof} Let us fix an arbitrary $\vecxi = (\xi_1,\dots,\xi_s, \dots) \in Y^\mN$ and denote by $\vecxi_s = (\xi_1,\dots,\xi_s, 0,0,\dots) \in Y^\mN$ its truncation.
Then Taylor’s Theorem implies 
	\begin{align*}
		F(\vecxi)
		=
		F(\vecxi_s)
		+
		\sum_{i\geq s+1}
		\partial^{\ve_j} F(\vecxi_s)\, \xi_j
		+
		\frac{1}{2!}
		\sum_{i\geq s+1}
		\sum_{j\geq s+1}
		\partial^{\ve_i+\ve_j} F(\vupsilon)\, \xi_i \xi_j,
	\end{align*}
	where $\ve_i$ is $i$-th vector unit and $\vupsilon\in\sett{\vupsilon\in Y^\mN: \exists t\in [0,1] \text{ s.t } \vupsilon=t\vecxi+(1-t)\vecxi_s}$ is a point between $\vecxi$ and $\vecxi_s$. Integrating both sides with respect to $\vecxi$ over $Y^\mN$ we obtain
	\begin{align*}
		I(F)-I_s(F)
		=
		\sum_{i\geq s+1}
		\int_{Y^\mN} 
		\partial^{\ve_j} F(\vecxi_s)\, \xi_j \, d \vecxi
		+
		\frac{1}{2!}
		\sum_{i\geq s+1
			\atop j\geq s+1}
		\int_{Y^\mN} 
		\partial^{\ve_i+\ve_j} F(\vupsilon)\, \xi_i \xi_j \, d \vecxi.
	\end{align*}
The first term in the right-hand side vanishes, since for every $j\geq s+1$ we have
	\begin{align*}
		\int_{Y^\mN} \partial^{\ve_j} F(\vecxi_s)\, \xi_j\,  d \vecxi
		=
		\int_{Y^s} \partial^{\ve_j} F(\xi_1,\dots,\xi_s,0,0,\dots)\, d \xi_1 \dots d \xi_s
		\, \int_{-\frac{1}{2}}^{\frac{1}{2}} \xi_j \, d \xi_j 
		=
		0.
	\end{align*}
	Noting \eqref{QMC lemma assumption} and $\abs{\xi_i}\leq \frac{1}{2}$ for all $i\in\mN$, we obtain
	\begin{align*}
		\abs{\int_{Y^\mN} 
			\partial^{\ve_i+\ve_j} F(\vupsilon)\, \xi_i \xi_j \,  d \vecxi}
		\leq
		\frac{1}{12}\,
		C_F
		\beta_i \beta_j
		(2!)^\delta.
	\end{align*}
	From this and the triangle inequality it follows that the error is bounded by
	\begin{align}\label{tail sum}
		\abs{I(F)-I_s(F)}
		\leq
		\frac{2^{\delta}}{24}
		C_F
		\brac{
			\sum_{j\geq s+1}
			\beta_j}^2.
	\end{align}
	Finally, in \cite[Theorem 5.1]{KuoSchwabSloan2012} it was shown that under \eqref{QMC lemma assumption} the tail of the
	sum over $\beta_j$ is bounded above by
	\begin{align*}
		\sum_{j\geq s+1}
		\beta_j\leq 
		\min\brac{\frac{p}{1-p},1}\norm{\vbeta}{\ell^p} s^{-\frac{1}{p}+1}.
	\end{align*}
	Substituting it into \eqref{tail sum} completes the proof.
\end{proof}

Denote $\{1:s\} = \{1, \dots, s\}$ and let $\bga = (\gamma_{\mfu})_{\mathfrak{u}\subseteq \{1:s\} }$ be a sequence of positive weights. We define the weighted Sobolev space of mixed first order derivatives 
$\mW_{\bga}(Y^s)$ as the collection of all functions $F : Y^s \mapsto \mR$ such that
\begin{align}\label{weighted norm def}
	\norm{F}{\mW_{\bga}(Y^s)}^2
	=
	\sum_{\mfu \subseteq \{1:s\} }
	\frac{1}{\gamma_{\mfu}} 
	\int_{Y^{\abs{\mfu}}}
	\brac{	\int_{Y^{\abs{\bar{\mfu}}}}
		\frac{\partial^{\abs{\mfu}} F}{\partial \vecxi_{\mfu}}(\vecxi)
		d\vecxi_{\bar{\mfu}}
	}^2
	d\vecxi_{\mfu} < \infty
\end{align}
Here $\bar{\mfu}:= \{1:s\}\setminus \mfu $ and $\frac{\partial^{\abs{\mfu}} F}{\partial \vecxi_{\mfu}}$ denotes the mixed first derivatives of $F$ with respect to the variable $\vecxi_{\mfu}=(\xi_j)_{j\in \mfu}$. The weight sequence $(\gamma_{\mfu})_{\mfu \subseteq \{1:s\}}$ is associated with each subset of the variables to moderate its relative importance with respect to the other subsets. With an appropriate choice of the weight we can derive an error bound, which is independent of the dimension~$s$. Moreover, we need some structure
of the weight for the component-by-component (CBC) construction cost to be feasible, see e.g. \cite{KuoNuyens2016,KuoSchwabSloan2012,KuoSchwabSloan2013}. Different types of weights have been considered depending on the problem and the estimation of $\frac{\partial^{\abs{\mfu}} F}{\partial \vecxi_{\mfu}}$.
We utilize the following particular result on the error of the QMC quadrature \eqref{QMC quad def} with $\vartheta\in \left(\tfrac{1}{2},1 \right]$, see, e.g., \cite[Theorem 4.1]{KuoSchwabSloan2013} and \cite[Theorem 2.1]{KuoSchwabSloan2012}
\begin{align}\label{QMC error}
	\sqrt{\mE\brac{\abs{I_s(F)-Q^{\Delta}_{s,n}(F)}^2}}
	\leq
	\brac{\sum_{\mfu\subseteq \{1:s\}}
		\gamma_{\mfu}^{\vartheta} \brac{\frac{2 \zeta(2\vartheta)}{(2\pi^2)^\vartheta} }^{\abs{\mfu}}
	}^\frac{1}{2\vartheta}
	\norm{F}{\mW_{\bga}(Y^s)}
	\varphi(n)^{-\frac{1}{2\vartheta}}.
\end{align}
Here $\mE$ denotes the expectation with respect to random shift $\Delta$, $\varphi(\cdot)$ denotes Euler's totient function and $\zeta(\cdot)$ denotes the Riemann zeta function. The following lemma contains the convergence estimate of the QMC quadrature \eqref{QMC quad def} for numerical integration of Gevrey-$\delta$ functions $F$ with $\delta \geq 1$.

\begin{lemma}\label{QMC theorem}
	Suppose the assumptions of Lemma \ref{general trunc theorem} are satisfied for the reduced range $p \in (0,\delta^{-1}]$ and consider the approximation of the integral \eqref{approx integral} by the randomly shifted lattice rule \eqref{QMC quad def} with $n=2^m$ quadrature points, $m \in \mathbb N$. When $p=\delta^{-1}$, assume additionally that $\|\vbeta\|_{\ell^p} < \sqrt{6}.$
	Then for
	\begin{align*}
		\vartheta = 
		\left\{\begin{matrix}
			\omega &\text{for some } \omega \in (\tfrac{1}{2},1), & \text{when } p\in (0,\tfrac{2}{3\delta}],
			\\
			\frac{\delta p}{2-\delta p},    &  & \text{when } p\in (\tfrac{2}{3\delta},\tfrac{1}{\delta}],
		\end{matrix}\right. 
	\end{align*}
	and the weights 
	\begin{align}\label{QMC weight}
		\gamma_{\mfu} = \brac{(\abs{\mfu}!)^{\delta} \prod_{j\in \mfu } \frac{ \beta_j}{\sqrt{\phi(\vartheta)}}}^{\frac{2}{1+\vartheta}},
		\text{ where }\quad
		\phi(\vartheta) = \frac{2 \zeta(2\vartheta)}{(2\pi^2)^\vartheta}
	\end{align}
	there exists a constant $C_{s,\gamma,\vartheta}$ independent of $n$ such that
	\begin{align}\label{equ: QMC error}
		\sqrt{\mE\brac{\abs{I_s(F)-Q^{\Delta}_{s,n}(F)}^2}}
		\leq
		C_{s,\gamma,\vartheta}^{\frac{1}{2}}\,\brac{\frac{n}{2}}^{-\frac{1}{2\vartheta}}.
	\end{align}
\end{lemma}
\begin{proof}
	The proof follows the lines of \cite[Theorem 4.2]{Gilbert2019} and \cite[Theorem 6.4]{KuoSchwabSloan2012}.
We start by estimating the weighted norm 	\eqref{weighted norm def} of a function $F$ satisfying \eqref{QMC lemma assumption} by
	\begin{align*}
		\norm{F}{\mW_{\bga}(Y^s)}^2 
		\leq
		C_F^2 \sum_{\mfu\subseteq  \{1:s\}} \frac{\Lambda_{\mfu}^2}{\gamma_{\mfu}}, \quad
		\text{ where }\quad
		\Lambda_{\mfu}:= (\abs{\mfu}!)^\delta \prod_{j\in \mfu} \beta_j.
	\end{align*}
	Notice that if $n=2^m$ is power of two, it holds that $\varphi(n) = \frac{n}{2}$. Therefore \eqref{QMC error} implies the estimate for the mean-square error
	\begin{align*}
		\mE\brac{\abs{I_s(F)-Q^{\Delta}_{s,n}(F)}^2}
		\leq
		C_{s,\gamma,\vartheta} \, \brac{\frac{n}{2}}^{-\frac{1}{\vartheta}}
	\end{align*}
	where
	\begin{align*}
		C_{s,\gamma,\vartheta}
		=
		\brac{\sum_{\mfu\subseteq \{1:s\}}
			\gamma_{\mfu}^{\vartheta}\, \phi(\vartheta)^{\abs{\mfu}}
		}^\frac{1}{\vartheta}
		\brac{C_F^2 \sum_{\mfu\subseteq  \{1:s\}} \frac{\Lambda_{\mfu}^2}{\gamma_{\mfu}} }.		 
	\end{align*}
Our aim is to demonstrate that $C_{s,\gamma,\vartheta}$ can be bounded independently of $s$.
Following \cite[Lemma 6.2]{KuoSchwabSloan2012} the optimal selection of the weights is given by \eqref{QMC weight}. For this specific weights we have 
	\begin{align*}
		C_{s,\gamma,\vartheta}
		= 
		C_F^2  \, S_{s,\vartheta}^{(1+\vartheta)/\vartheta}, \quad \text{where} \quad
		S_{s,\vartheta}
		:=
		\sum_{\mfu \in \{1:s\}}
		\brac{\Lambda_{\mfu}^{2\vartheta}\, \phi(\vartheta)^{\abs{\mfu}} }^{\frac{1}{1+\vartheta}}.
	\end{align*}
	Thus, it remains to demonstrate that $S_{s,\vartheta}$ can be bounded independently of $s$. 
	In the case $p=\delta^{-1}$, we select $\vartheta=1$, and recall that $\phi(1)=\tfrac{1}{6}$. Since $\delta\geq 1$, we have
	\begin{align*}
		S_{s,\vartheta}
		=
		\sum_{\mfu \subseteq \{1:s\}}  (\abs{\mfu}!)^{\delta}  \prod_{j \in \mfu} 
		\frac{\beta_j }{\sqrt{6}}
		\leq
		\brac{\sum_{\mfu \subseteq \{1:s\}}  (\abs{\mfu}!)  \prod_{j \in \mfu} 
		\brac{\frac{\beta_j }{\sqrt{6}}}^{\frac{1}{\delta}}}^{\delta}
		\leq
		\brac{1-\sum_{j\geq 1} \brac{\frac{\beta_j }{\sqrt{6}}}^{\frac{1}{\delta}}}^{-\delta}\hspace*{-3mm},
	\end{align*}
where we have used \cite[Lemma 6.3]{KuoSchwabSloan2012} in the last step. The right-hand side is finite, since $\|\vbeta\|_{\ell^p} \leq \sqrt{6}$, and is independent of $s$. 

For the remaining case $0<p<\delta^{-1}$ we recall again that $\delta \geq 1$ and thereby obtain
	\begin{align*}
		&S_{s,\vartheta}
		=
		\sum_{\mfu \in \{1:s\}}
		\brac{\Lambda_{\mfu}^{2\vartheta}\, \phi(\vartheta)^{\abs{\mfu}} }^{\frac{1}{1+\vartheta}}
		\leq
		\brac{
			\sum_{\mfu \in \{1:s\}}
			\brac{\Lambda_{\mfu}^{2\vartheta}\, \phi(\vartheta)^{\abs{\mfu}} }^{\frac{1}{(1+\vartheta)\delta}}
		}^{\delta}\\
		&=
		\brac{
			\sum_{\mfu \in \{1:s\}}
			(\abs{\mfu}!)^{\frac{2 \vartheta }{1+\vartheta}}  \prod_{j \in \mfu} 
			\beta_j^{\frac{2\vartheta}{(1+\vartheta)\delta}} 
			\phi(\vartheta)^{\frac{1}{(1+\vartheta)\delta}}
		}^{\delta}
		=
		\brac{
			\sum_{\mfu \subseteq \{1:s\}}  
			(\abs{\mfu}!)^{\varrho}  \prod_{j \in \mfu} 
			\beta_j^{\frac{\varrho}{\delta}} 
			\phi(\vartheta)^{\frac{1}{(1+\vartheta)\delta}} 
		}^\delta
	\end{align*}
	where $\varrho:= \frac{2 \vartheta }{1+\vartheta} < 1$ for $\vartheta < 1$.
Let 
	$\eta:= \frac{\varrho }{1-\varrho} =
	\frac{2\vartheta}{1-\vartheta}$ 
	and
	$ \kappa:=
	\frac{1}{(1+\vartheta)\delta(1-\varrho)}
	=
	\frac{1}{(1-\vartheta)\delta}$. We multiply and divide each term in the above estimate by $\prod_{j \in \mfu} \alpha_j^{\varrho}$ with $\alpha_j > 0$ to be specified 
	later. Then we apply H\"older’s inequality with conjugate exponents $\frac{1}{\varrho}>1$ and $\frac{1}{1-\varrho}>1$ and \cite[Lemma 6.3]{KuoSchwabSloan2012} to obtain
	\begin{equation} \label{Sstheta}
	\begin{split}
		S_{s,\vartheta}
		&\leq
		\brac{
			\sum_{\mfu \subseteq \{1:s\}}  (\abs{\mfu}!)^{\varrho} 
			\brac{\prod_{j \in \mfu} \alpha_j^\varrho}
			\brac{ \prod_{j \in \mfu} 
				\brac{\frac{\beta_j^{\frac{1}{\delta}}}{\alpha_j}}^{\varrho} 
				\phi(\vartheta)^{\frac{1}{(1+\vartheta)\delta}}}}^\delta
		\\
		&\leq
		\brac{\sum_{\abs{\mfu} < \infty}  \abs{\mfu}! \prod_{j \in \mfu} \alpha_j}^{\varrho\delta}
		\brac{\sum_{\abs{\mfu} < \infty}   \prod_{j \in \mfu} 
			\brac{\frac{\beta_j^{\frac{1}{\delta}}}{\alpha_j}}^{\eta}  \phi(\vartheta)^{\kappa} }^{{(1-\varrho)\delta}}
		\\
		&\leq
		\brac{\frac{1}{1-\sum_{j\geq 1} \alpha_j}}^{\varrho\delta} 
		\exp \brac{(1-\varrho)\delta\,{\phi(\vartheta)^{\kappa}}
			\sum_{j\geq 1} 	\brac{\frac{\beta_j^{\frac{1}{\delta}}}{\alpha_j}}^{\eta}}.
			\end{split}
	\end{equation}
The right-hand side is finite if $\sum_{j\geq 1}\alpha_j < 1$ and $\sum_{j\geq 1} 	\big( \beta_j^{\frac{1}{\delta}}/\alpha_j\big)^{\eta}$ is finite. In this case this bound is also independent of $s$. We now choose
	\begin{align*}
		\alpha_j := \frac{\beta_j^{p}}{\tau} \text{ for some } \tau > \norm{\vbeta}{\ell^p}^p.
	\end{align*}
This clearly implies the first condition  $\sum_{j\geq 1}\alpha_j < 1$. Concerning the second condition, we observe that
	\[
\sum_{j\geq 1} 	\brac{\frac{\beta_j^{\frac{1}{\delta}}}{\alpha_j}}^{\eta} 	
= \tau^\eta \sum_{j\geq 1} 	\beta_j^{(\frac{1}{\delta}-p)\eta}. 
	\] 
According to \eqref{QMC lemma assumption}, this sum is finite if $(\frac{1}{\delta}-p)\eta \geq p$. Recalling the definition of $\eta$, this is equivalent to
	\begin{align*}
\vartheta \geq \frac{\delta p}{2-\delta p}.
	\end{align*}
	Since the error estimate \eqref{QMC error} and \eqref{Sstheta} are simultaneously valid for $\vartheta \in (\frac{1}{2},1)$, we choose $\vartheta$ as any fixed value $\omega\in\left(\tfrac{1}{2}, 1 \right)$ when $\delta p\in \left( 0,\tfrac{2}{3}  \right]$, and for $\delta p\in \left(\tfrac{2}{3}  , 1  \right)$, we set $\vartheta = \frac{\delta p}{2-\delta p}$. This finishes the proof.
\end{proof}
\begin{remark}\label{QMC remark}
	Theorem \ref{QMC theorem} naturally extends \cite[Theorem 4.2]{Gilbert2019} and \cite[Theorem 6.4]{KuoSchwabSloan2012} to the case $\delta > 1$.
	\end{remark}

\end{document}